\documentclass[11pt]{article}
\usepackage{authblk}
\usepackage[T1]{fontenc}
\usepackage[utf8]{inputenc}
\usepackage[english]{babel}

\usepackage{lmodern}
\usepackage{mathtools}
\usepackage[a4paper,DIVcalc]{typearea}
\typearea[10mm]{16}

\usepackage{amsthm}
\usepackage{dsfont}
\usepackage{latexsym}
\usepackage[pdftex]{graphicx}
\usepackage{mathdots}
\usepackage[hyperindex=true, pdftex=true, linkcolor=black,colorlinks=true,citecolor=blue,urlcolor=blue]{hyperref}
\usepackage{amssymb,amsfonts,textcomp}

\usepackage{xcolor}
\definecolor{LightGray}{gray}{0.9}

\usepackage{pgf, tikz}
\usepackage{authblk}

\title{Finite volume scheme and renormalized solutions for nonlinear elliptic
  Neumann problem with $L^1$ data}
\author{Mirella AOUN}
\author{Olivier GUIBÉ}
\affil{Laboratoire de Math\'ematiques Rapha\"el Salem, University of Rouen, UMR CNRS 6085, Avenue de l'Universit\'e, 76801 Saint Etienne du Rouvray, France\\
Email: \texttt{mirella.aoun@univ-rouen.fr, olivier.guibe@univ-rouen.fr}}
\date{\today}

\newtheorem{thm}{Theorem}[section]
\newtheorem{defin}[thm]{Definition}
\newtheorem{prop}[thm]{Proposition}
\newtheorem{lemma}[thm]{Lemma}
\newtheorem{remark}[thm]{Remark}
\newtheorem{corollary}[thm]{Corollary}
\numberwithin{equation}{section}
\theoremstyle{remark}

\newenvironment{proof2}{\noindent\textbf{Proof~}}{\hfill$\square$\bigbreak}

\newcommand{\med}{\operatorname{med}}

\newcommand{\medu}{\underline{\operatorname{med}}}
\newcommand{\medo}{\overline{\operatorname{med}}}
\newcommand{\card}{\operatorname{card}}
\newcommand{\dx}{\operatorname{d}\!x}
\newcommand{\meas}{\operatorname{meas}}
\newcommand{\diw}{\operatorname{div}}

\newcommand\R{\mathbb{R}}
\newcommand\N{\mathbb{N}}
\newcommand\mA{\mathbb{A}}

\newcommand\E{\mathcal{E}}

\newcommand\tq{\, ; \ }
\newcommand\T{\mathcal{T}}

\newcommand\A{\mathcal{A}}
\newcommand\B{\mathcal{B}}

\newcommand\M{\mathcal{M}}

\newcommand\xx{\boldsymbol{x}}
\newcommand\vv{\boldsymbol{v}}

\newcommand\n{\boldsymbol{
\mathrm{n}}_{K,\sigma}}
\newcommand\um{u^m}

\newcommand\m{\mathrm{m}}
\newcommand\Snbm{\overline{S}^m_n}
\newcommand\labm{\overline{\lambda}^m}
 \allowdisplaybreaks[1]

\begin{document}

\maketitle

\begin{abstract}
In this paper we study the convergence of a finite volume approximation of a convective diffusive elliptic problem with Neumann boundary conditions and $L^1$ data. To deal with the non-coercive character of the equation and the low regularity of the right hand-side we mix the finite volume tools and the renormalized techniques. To handle the Neumann boundary conditions we choose solutions having a null median and we prove a convergence result.
\end{abstract}

%\begin{itemize}
   % \item rédiger introduction
    %\item remarque moyenne =0 : 2 août 2021 {\color{red} à compléter, biblio, peut-être un énoncé de théorème}
     %\item remplacer $h$ par $S$ ($h_n$ par $S_n$ ou $\theta_n$ {\color{blue} c'est fait}
    %\item augmenter la biblio {\color{blue} j'ai augmenté un peu}
    %\item faire la somme arêtes intérieures
    %\item homogénéité notation (u tilde, $u_M$...) {\color{blue} aussi dans la preuve de convergence }
    %{\color{blue}\item dans la preuve de convergence, question sur les $w(n)$,}
    %\item {\color{blue} $u_{\M_{m}}$ et $h_{\M_{m}}$ verifier partout notation}.
    %\item {\color{blue} voir aussi la puissance $u^{m}$}
    %\item relire

%\end{itemize}

\section{Introduction}
In the present paper we are interested in the discretization by the cell-centered finite volume method of the following convection-diffusion equation with Neumann boundary conditions and $L^1$ data:
\begin{equation}\label{Pb}
    \begin{aligned}{}
   & & -\diw(\lambda(u)\nabla u -\vv u)=f &~~~ \text{in $\Omega$},
\\& &
(\lambda(u)\nabla u -\vv u)\cdot\vec{n}=0 &~~~\text{on $\partial\Omega$}.
    \end{aligned}
\end{equation}
%\label{assumptions}
Here $\Omega$ is a bounded polygonal connected open subset of $\mathbb{R}^{d}$, $d\geq 2$, $\vec{n}$ is the outer unit normal to $\partial \Omega$ and $\lambda$ is a continuous function such that $\lambda_{\infty}\geq \lambda(u)\geq \mu >0$ with $\lambda_\infty$ and $\mu$ two real numbers. The function $\vv$ lies in $ L^{p}(\Omega)^d$ with $2<p<+\infty$ if $d=2$, $p=d$ if $d \geq 3$, and $f$ belongs to $L^1(\Omega)$ and satisfies the compatibility condition $\int_{\Omega}f=0$.

Considering elliptic equations with $L^1$ data requires a precise meaning of
solution. Indeed we cannot expect in general to obtain a usual weak solution
which belongs  to $H^1_0(\Omega)$ for Dirichlet boundary conditions or to
$H^1(\Omega)$ for Neumann boundary conditions. Elliptic equations with $L^1$
data and Dirichlet boundary conditions are widely studied in the literature. In
\cite{BG89} Boccardo and Gallouët have obtained the existence of a solution in
the sense of distributions for a fairly class of monotone operator with measure
data. However it is known that this solution is not unique in general (see the
counter example of Serrin \cite{Serrin}). To overcome the lack of uniqueness
results, it is possible to use in the linear case the duality method (see \cite{St})
or, for general nonlinear operators, the notion of entropy solution (see \cite{zbMATH01463265}),
the notion of solution obtained as limit of approximation (SOLA) (see \cite{AD}) or the notion
of renormalized solution (see \cite{Mu2,zbMATH01463265}). The previous three notions of solution are
equivalent in the $L^1$ case and provide existence, stability and uniqueness
results for a large class of elliptic equations.
As far as the approximation of elliptic equations with Dirichlet boundary
conditions and $L^1$ data is concerned, the method of finite volume (see \cite{EGH2000}) allows to
consider such equations. In \cite{zbMATH02120443} the authors have studied equation \eqref{Pb}
with $\vv=0$ and with a measure data (and Dirichlet boundary conditions). In \cite{DGH03}
the authors have considered a linear noncoercive equation (similar to \eqref{Pb}) with measure
data (and Dirichlet boundary conditions). In both papers \cite{zbMATH02120443,DGH03} the authors
have established the convergence of the finite volume approximation to a
solution in the sense of distributions. More precisely for the equation $-\Delta
u + \diw(\vv u)=f$ in $\Omega$ with Dirichlet boundary conditions, the limit $u$ of
the finite volume scheme verifies
\[
  \left\{
    \begin{aligned}
      & u\in \bigcap_{q<d/(d-1)} W^{1,q}_{0}(\Omega) \\
      & \int_{\Omega} \nabla u \nabla \varphi \dx -\int_{\Omega} u\vv \nabla
      \varphi \dx = \int_{\Omega} f\varphi \dx,\quad \forall \varphi\in
      \bigcup_{s>d} W^{1,s}_{0}(\Omega).
    \end{aligned}
    \right.
\]
Recently mixing the techniques of renormalized solution and the finite volume
approximation has been performed in \cite{leclavier} for a noncoercive equation with
$L^{1}$ data and Dirichlet boundary conditions: the author proves that the limit
of the finite volume scheme is the renormalized solution of the equation.
Concerning  the finite elements approximation  the
model case  of the equation $-\diw(A\nabla u)=f$ with Dirichlet boundary
conditions is dealt in \cite{zbMATH05119704}.

In the present paper we have to face to a noncoercive equation, to an $L^{1}$
data and to Neumann boundary conditions. To our knowledge such a situation is less studied
in the literature both in the continuous case and the discrete case. One of the
difficulty  in the variational and linear case is that the kernel is nontrivial
and that we have to impose an additional condition on the solution to
insure uniqueness result, which is in general $\int_{\Omega} udx=0$.
In \cite{DV09} by using the Fredholm theory the authors have been studied the
operator associated to the linear version of  \eqref{Pb}. They prove that the
linear version of \eqref{Pb} with $(H^{1})'$ data verifying a compatibility
condition admits a unique weak solution. Moreover they deduce existence and uniqueness results
 for elliptic and coercive equation of the type
$-\diw(A(x,u)\nabla u))=\mu$ with Neumann boundary condition, where $\mu$ is
a bounded Radon measure. The finite volume approximation  of \eqref{Pb} with $f$
belonging to $L^{2}(\Omega)$ (with zero mean value) is studied in
\cite{CHD11}. As in the continuous case
the finite volume approximation requires the study of the kernel and for
different approximations of the convective terms the authors prove in
\cite{CHD11} that the finite volume approximation converges to a weak solution
of \eqref{Pb}. For the class of  nonlinear elliptic equations $-\Delta_{p}u=f$ with Neumann
boundary conditions, $L^{1}$ data and for small value of $p$ it is well known
that the solution is not in general a summable function so that the mean value
has no meaning. To overcome this obstacle, in \cite{ACMM,BGM1} the authors have chosen
the median value which is well defined instead of the mean value. In \cite{BGM1}
an appropriate definition of renormalized solutions is given, which gives an
existence result (see also \cite{BGM2} for the uniqueness question).
The main originality of the present paper is to consider noncoercive equation \eqref{Pb}
with $L^{1}$ data and to mix the techniques developed in \cite{BGM1} and the
finite volume method. We choose here  the median value instead of the mean value
as in \cite{CHD11}.  Since Poincar\'e-Wirtinger inequality is crucial in general we state in Proposition \ref{prop2.9} an appropriate discrete Poincar\'e-Wirtinger inequality
involving the median value (see Appendix for the proof, in the spirit of \cite{zbMATH06476912}).
In Theorem \ref{conv} we prove that the
finite volume approximation of \eqref{Pb} converges to the renormalized solution with a null
median.

\smallskip
The paper is organized as follows. In Section 2 we recall some definitions, in
particular the median of a measurable function. Moreover we present in Section 2
the continuous case and the notion of renormalized solution of \eqref{Pb} and, at
last the finite volume tools and the scheme. The main results are stated in
Section 3. Section 4 is devoted to derive the a priori estimates for the
solutions of the scheme. Using Section 4 we prove the existence of a solution of
the scheme in Section 5 while the convergence analysis is performed in Section
6. Finally we give in Appendix the proof of the discrete Poincar\'e-Wirtinger
inequality involving the median (instead of the mean value).

\section{Assumptions and definitions}
Let $\Omega$ be a connected open bounded polygonal subset of $\mathbb{R}^d$, $d\geq 2$. We consider the following nonlinear elliptic problem with Neumann boundary conditions:
\begin{equation}\label{Pb1}
\left\{    \begin{aligned}{}
   & & -\diw(\lambda(u)\nabla u -\vv u)=f &~~~ \text{in $\Omega$},
\\& &
(\lambda(u)\nabla u -\vv u)\cdot\vec{n}=0 &~~~\text{on $\partial\Omega$},
    \end{aligned} \right.
\end{equation}
where $\vec{n}$ is the outer unit normal to $\partial \Omega$. We assume that
\begin{align}\label{assump1}
     &\vv \in L^{p}(\Omega)^d\ \textnormal{with}\ 2<p<+\infty \ \textnormal{if}\ d=2,  p=d\ \textnormal{if}\  d \geq 3 , \\
&\lambda\ \textnormal{is a continuous function such that}\ \lambda_{\infty}\geq \lambda(r)\geq \mu >0, \forall r\in\R , %\textnormal{with}\\ &\lambda_\infty\ \textnormal{and}\ \mu\ \textnormal{are two real numbers}.
\end{align}
%function $v$ is in $ L^{p}(\Omega)^d$ with $2<p<+\infty$ if $d=2$, $p=d$ if $d \geq 3$, and $\lambda$ is a continuous function such that $\lambda_{\infty}\geq \lambda(u)\geq \mu >0$
with $\lambda_\infty$ and $\mu$ two real numbers. Moreover, we assume that
\begin{align}
    f \in L^1(\Omega),
\end{align}and it satisfies the compatibility condition
\begin{align}\label{assump2}
    \int_\Omega f \,\mathrm{d}x=0.
\end{align}
 As explained in the Introduction we deal with solutions whose median is equal to zero. Let us recall that if $u$ is measurable function, we define the median of $u$ (with respect to the Lebesgue measure), denoted by $\med(u)$ as the set of real numbers $t$ such that
 \begin{align*}
     &\meas  \{x\in\Omega:u(x)>t\}\leq\dfrac{\meas(\Omega)}{2}\\
     &\meas  \{x\in\Omega:u(x)<t\}\leq\dfrac{\meas(\Omega)}{2}.
 \end{align*}
  It is known that $\med(u)$ is non-empty compact interval (see \cite{Z}). Let us explicitly observe that if $0\in \med(u)$ then \begin{align*}
     &\meas  \{x\in\Omega:u(x)>0\}\leq\dfrac{\meas(\Omega)}{2}\\
     &\meas  \{x\in\Omega:u(x)<0\}\leq\dfrac{\meas(\Omega)}{2}.
 \end{align*}
 We denote $\underline{\med}(u)$ by
 \begin{align}
     & \underline{\med}(u)= \inf \left\{t \in \mathbb{R}:\meas\{x\in\Omega:u(x)>t\}\leq\dfrac{\meas(\Omega)}{2}\right\},
 \end{align}
 and $\overline{\med}(u)$ by
  \begin{align}
     & \overline{\med}(u)= \sup \left\{t \in \mathbb{R}:\meas\{x\in\Omega:u(x)>t\}\geq\dfrac{\meas(\Omega)}{2}\right\}.
 \end{align}

We observe that if $u$ is an element of $H^1(\Omega)$ ($\Omega$ being a connected domain), the median of $u$ is uniquely determined; $\med(u)=\underline{\med}u=\overline{\med}(u)$. However it is not the case for the finite volume approximation of \eqref{scheme} which is a piecewise-constant function; the median is then the compact interval of $\R$ $[\medu(u),\medo(u)]$.
 %It is known that $\med(u)$ is a non-empty compact interval (see \cite{z}).
 %\begin{defin}
 %\begin{equation}
   %5 \med(u)=\sup\left\{t \in \mathbb{R} : \meas  %\{x\in\Omega:u(x)>t\}>\dfrac{\meas(\Omega)}{2}\rig%ht\}.
%\end{equation}
 %\end{defin}

In the whole paper, $T_n$, $n\geq 0$, denotes the truncation at height $n$ that is \[T_n(s)=\min(n,\max(s,-n)), \quad\forall s \in \mathbb{R}.\]

\subsection{Continuous Case}

In this subsection we precise the notion of solution of equation \eqref{Pb1}. Indeed as explained in the Introduction,  considering elliptic equations with $L^1$ data requires an appropriate notion of solution which provides existence, stability and uniqueness results. There is a wide literature in the Dirichlet case.  In the Neumann case, due to the lack of regularity of the solution, the mean value may not exist for nonlinear problems with $L^1$ data, which gives additional difficulties in deriving estimates and in defining an appropriate notion of renormalized solution. We refer mainly to \cite{D00} for linear problems using the duality method and  to \cite{Prignet97}, \cite{ACMM} and \cite{BGM1} for nonlinear problems.  In \cite{ACMM} and \cite{BGM1} the authors have chosen the median instead of the mean value (which may not exist if the solution is not integrable) and one of the main tool is the following Poincaré-Wirtinger inequality, see \cite{Z}.
\begin{prop}\label{propPW}
If $u \in W^{1,p}(\Omega)$, then
\begin{equation}
   \|u-\med(u)\|_{L^p(\Omega)} \leq C\|\nabla u\|_{(L^p(\Omega))^d}
\end{equation}
where $C$ is a constant depending on $p$, $d$, $\Omega$.
\end{prop}

In \cite{BGM1} the authors prove the existence of a renormalized solution for a class of nonlinear problems and prove in \cite{BGM2} uniqueness results under additional assumptions.  In the present paper we use the framework of renormalized solutions. In the particular case of equation \eqref{Pb1}, let us recall the following definition (see \cite{BGM1}).
\begin{defin} \label{def}
A real  function $u$ defined in $\Omega$ is a renormalized solution to \eqref{Pb1} if
\begin{gather}
\label{def0}u \text{ is measurable and finite almost everywhere in $\Omega$,}
\\
\label{def1}
     T_n(u) \in H^1(\Omega),\text{ for any } n>0,
\\
\label{def2}
   \lim\limits_{n \rightarrow +\infty}\frac{1}{n} \int_{\{x\in\Omega,|u(x)|<n\}} \lambda(u) |\nabla u |^2 \, \mathrm{d}x=0,
\end{gather}
and the following equation holds
\begin{equation}\label{def3}
     \begin{aligned}
 &\int_\Omega S(u) \lambda(u)  \nabla u\cdot \nabla
 \varphi\,\mathrm{d}x +
 \int_\Omega S'(u) \lambda(u) \varphi\nabla u \cdot \nabla u\,\mathrm{d}x
 \\&{} -\int_\Omega u S(u)\vv \cdot \nabla\varphi \, \mathrm{d}x
 - \int_\Omega u  S'(u) \varphi  \vv \cdot \nabla u  \,\mathrm{d}x
 = \int_\Omega f \varphi S(u)\,\mathrm{d}x,
 \end{aligned}
\end{equation}
for every $S \in W^{1,\infty}(\mathbb{R}) $ having compact support and for every $ \varphi \in L^\infty(\Omega)\cap H^{1}(\Omega)$.
\end{defin}

By combining \cite{BGM1} and \cite{BGM2}  we have the following existence and uniqueness result.

\begin{thm}
Let us assume that \eqref{assump1}--\eqref{assump2} hold true. Then there exists a unique renormalized solution $u$ of \eqref{Pb1} such that $\med(u)=0$.
\end{thm}

\begin{remark} As far as the uniqueness is concerned equation \eqref{Pb1} is not directly in the scope of \cite{BGM2}. Indeed uniqueness results are mainly obtained for equations whose prototype is $-\diw(a(x,\nabla u)+\Phi(x,u))=f$ with Neumann boundary conditions. The operator $a(x,\nabla u)$ does not depend on $u$. Due to the presence of $\lambda(u)$ in equation \eqref{Pb1} the quasilinear character allows one to obtain the uniqueness by a changement of unknow. Since $\lambda(r)$ is a continuous function such that $\lambda_\infty\geq \lambda(r)\geq\mu>0$, by defining $\widetilde{\lambda}(r)=\int_0^r\lambda(s) \mathrm{d}s$ and $w=\widetilde{\lambda}(u)$, we can verify that the function $w$ has a null median and that $w$ is a renormalized solution of
\begin{equation}\label{Pbis}
\left\{    \begin{aligned}{}
   & & -\diw(\nabla w -\vv \widetilde{\lambda}^{-1}(w))=f &~~~ \text{in $\Omega$},
\\& &
(\nabla w -\vv \widetilde{\lambda}^{-1}(w))\cdot\vec{n}=0 &~~~\text{on $\partial\Omega$}.
    \end{aligned}\right.
\end{equation}
At last since the function $\widetilde{\lambda}^{-1}$ is Lipschitz continuous, Theorem 4.2 of \cite{BGM2} allows one to conclude that $w$ is unique so that $u$ is unique.
\end{remark}

\subsection{Finite Volume}
We now introduce the discrete settings. Let us first recall the notion of admissible discretization of $\Omega$ , the definitions of the discrete norms and the space of piecewise functions associated to an admissible mesh following \cite{EGH2000}.
%We now introduce the discrete settings. We first present the notion of admissible discretization of the domain $\Omega$ assumptions on the mesh and and discrete norms following \cite{EGH2000}.

  \begin{defin}[Admissible mesh]\label{admi} An admissible mesh $\mathcal{M}$ of
    $\Omega$ is given by a finite family $\mathcal{T}$ of disjoint open convex
    polygonal subsets of $\Omega$, a finite family $\mathcal{E}$ of disjoint
    subsets of $\Bar{\Omega}$ (the edges) consisting in non-empty open convex
    subsets of affine hyperplanes and a family $\mathcal{P}=(x_K)_{K\in
      \mathcal{T}}$ of points in $\Omega$ such that
    \begin{itemize}
    \item $\Bar{\Omega}=\cup_{K\in \mathcal{T}}\Bar{K}$,
    \item each $\sigma \in \mathcal{E}$ is a non-empty open subset of $\partial K$ for some $K \in \mathcal{T}$,
    \item by denoting $\mathcal{E}_K=\{\sigma \in \mathcal{E}, \sigma \subset \partial K \}$, $\partial K=\cup_{\sigma \in \mathcal{E}_K}\sigma$ for all $K \in \mathcal{T}$,
    \item for all $K \ne L$ in $\mathcal{T}$, either the $(d-1)-$dimentional measure of $\Bar{K}\cap \Bar{L}$ is zero or $\Bar{K}\cap \Bar{L}=\Bar{\sigma}$ for some $\sigma \in \mathcal{E}$, which is then denoted $\sigma=K|L$,
    \item for all $K\in\mathcal{T}$, $x_K \in K$,
    \item for all $\sigma=K|L \in \mathcal{E}$, the straight line $(x_K,x_L)$ intersects and is orthogonal to $\sigma$,
    \item for all $\sigma \in \mathcal{E}$ such that $\sigma \subset \partial \Omega \cap \partial K$,  the line which is orthogonal to $\sigma$ and goes through $x_K$ intersects $\sigma$.
    \end{itemize}

 \end{defin}
In the whole of the present paper, we use the following notations associated
with an admissible discretization. In the set of edges $\mathcal{E}$, we
distinguish the set of interior edges $\mathcal{E}_{int}$ and the set of
boundary edges $\mathcal{E}_{ext}$. We denote by m($K$) the $d$-dimensional
measure of a control volume $K$ and m($\sigma$) the $(d-1)$-dimensional measure
of $\sigma$. For all $\sigma \in \mathcal{E}_K$, $\boldsymbol{
  \mathrm{n}}_{K,\sigma}$
is the unit normal to $\sigma$ outwards $K$. If $\sigma=K|L \in
\mathcal{E}_{int}$, we denote by $d_\sigma$ the Euclidian distance between $x_K$
and $x_L$, $d_\sigma=d_{K,\sigma}+d_{L,\sigma}$ and $d_\sigma=d_{K,\sigma}$ if
$\sigma \in \mathcal{E}_{ext}\cap\mathcal{E}_{K}$.\\
  The size of the mesh is defined by
  $$h_\mathcal{M}=\mathrm{sup}_{K\in\mathcal{T}}\mathrm{diam}(K).$$\\
  We assume that the mesh satisfies the following assumption
  \begin{equation}\label{cmesh}
      \exists \ \xi>0\; \text{such that }d(x_K,\sigma)\geq \xi d_\sigma,~~\forall {\mathcal T}\in \mathcal{E}, \forall \sigma \in \mathcal{E}_K.
  \end{equation}
 An example of admissible mesh in the sense of the above definition is shown in Figure \ref{fig:D_sigma}.

 The space of piecewise functions associated to an admissible mesh, denoted by
 $X(\mathcal{M})$, is defined as the set of functions from $\Omega$ to $\mathbb{R}$ wich are constant over each control volume of the mesh.

  %Let us now define some discrete norms and seminormes on this space.
  \begin{defin}[Discrete $W^{1,p}$ norm] Let $\Omega$ be an open bounded
    polygonal subset of $\mathbb{R}^d$, $d\geq 2$, and let $\mathcal{M}$ be an
    admissible mesh. For $u=(u_{K})_{K\in\mathcal{T}}\in X(\mathcal{T})$ and $p \in [1,+\infty[$, the
    discrete $W^{1,p}$-semi-norm is defined by
     $$|u|_{1,p,\mathcal{M}}=\left(\sum_{\underset{\sigma=K|L}{\sigma\in\mathcal{E}_{int}}}\frac{\m(\sigma)}{d_{\sigma}^{p-1}}\left|{u_K-u_L}\right|^p\right)^{\frac{1}{p}},\hspace*{1cm}
     \forall u\in X(\mathcal{T})$$
and the discrete $W^{1,p}$-norm is defined by
\[
  \|u\|_{1,p,\mathcal{M}}=\|u\|_{0,p}+|u|_{1,p,\mathcal{M}},
  \hspace*{2cm}\forall u \in X(\mathcal{T})
\]
where $\|u\|_{0,p}$ is the $L^p$
norm for piecewise constant functions, $\forall p\in[1,+\infty[$,
\[\
  \|u\|_{0,p}= \left (\int_\Omega|u(x)|^{p}\mathrm{d}x \right
)^{\frac{1}{p}}=\left( \sum_{K\in
    \mathcal{M}}\m(K)|u_K|^p\right)^{\frac{1}{p}},\;\;\forall u \in
X(\mathcal{T}).
\]
  \end{defin}
  We present now discrete functional analysis results. We refer the reader to
  [\cite{CHD11}, Lemma 6.1] for a proof of the following discrete Sobolev
  inequality. %The proof of the following discrete Sobolev inequality can be
              %found in [\cite{CHD11}, Lemma 6.1].
 \begin{prop}[Discrete Sobolev inequality]
 Let $\Omega$ be a bounded polygonal open subset of $\R^d$ and let $\M$ be an
 admissible mesh satisfying \eqref{cmesh}. Let $q<+\infty$ if $d=2$ and
 $q=\frac{2d}{d-2}$ if $d\geq 3$. Then there exists $C=C(\Omega,\xi,q)$ such
 that, for all $u=(u_K)_{K\in\T}\in X(\mathcal{T})$,
 \begin{equation}\label{Sobolev}
     \|u\|_{0,q}\leq C\left(|u|_{1,2,\M}+\|u\|_{0,2}\right).
 \end{equation}
 \end{prop}

In the already cited references discrete Poincaré and Poincaré-Wirtinger
inequalities are related to the discrete space $W^{1,p}_0(\Omega)$ and zero
boundary condition or the discrete space $W^{1,p}(\Omega)$ with discrete mean
value. We derive here a discrete Poincaré-Wirtinger inequality involving the
median. The proof is given in the appendix.
\begin{prop}[Discrete Poincaré-Wirtinger median inequality]\label{prop2.9}
  Let $\Omega$ be an
  open bounded connected polyhedral domain of $\R^d$ and let $\mathcal{M}$ be an
  admissible mesh satisfying \eqref{cmesh}. Then for $1\leq p <+\infty
  $ there exists a constant $C>0$ only depending on $\Omega$, $d$ and $p$ such that
  \begin{equation}\label{PW}
      \|u-c\|_{0,p}\leq \dfrac{C}{\xi^{(p-1)/p}}|u|_{1,p,\mathcal{M}}, \hspace*{1cm}\forall u\in X(\mathcal{T})
  \end{equation}{}
  where $c$ belongs to  $\med(u)$.
  \end{prop}

 \begin{thm}[Discrete Rellich's theorem]\label{compact}
Let $(\M_m)_{m\ge1}$ be a sequence of admissible meshes satisfying \eqref{cmesh}
and such that $h_{\M_m}\rightarrow0$ as $m \rightarrow \infty$. If $v_{m} \in
X(\T_m)$ is such that $(|v_m|_{1,2,\M}+\|v_m\|_{0,2})$ is bounded, then
$(v_m)_{m\in \mathbb{N}}$ is relatively compact in $L^2(\Omega)$. Furthermore,
any limit in $L^2(\Omega)$ of a subsequence of $(v_m)_{m\in \mathbb{N}}$ belongs
to $H^1(\Omega)$.
\end{thm}

Let us now define a discrete finite volume gradient introduced equivalently in [\cite{zbMATH01970883}, Lemma 4.4], [\cite{zbMATH05490416}, Lemma 6.5] or [\cite{zbMATH02027732}, Definition 2].
 \begin{defin}[Discrete finite volume gradient]
For $K\in\mathcal{M}$ and $\sigma\in\mathcal{E}(K)$, we define the volume $D_{K,\sigma}$ as the cone of basis $\sigma$ and of opposite vertex $x_K$.Then, we define the "diamond-cell" $D_{\sigma}$ (see Figure \ref{fig:D_sigma}) by
\begin{align*}
D_{\sigma} &= D_{K,\sigma}\cup D_{L,\sigma}\ &&\textnormal{if}\ \sigma=K|L\in\mathcal{E}_{int}, \\
D_{\sigma} &= D_{K,\sigma}\ &&\textnormal{if}\ \sigma\in\mathcal{E}_{ext}\cap\mathcal{E}_{K},
\end{align*}
and $$\m(D_\sigma)=\frac{1}{d}d_\sigma \m(\sigma).$$

The approximate gradient $\nabla_{\mathcal{M}}u$ of a function $u \in X(\mathcal{T})$ is defined as a piece-wise constant function over each diamond cell and given by
\begin{align*}
&\forall \sigma\in\mathcal{E}_{int},\ \sigma=K|L,\ && \nabla_{\mathcal{M}}u(\xx)=d\frac{u_L-u_K}{d_\sigma}\boldsymbol{
  \mathrm{n}}_{K,\sigma},\ & \forall \xx\in D_{\sigma}, \\
&\forall \sigma\in\mathcal{E}_{ext}\cap\mathcal{E}_{K}, && \nabla_{\mathcal{M}}u(\xx)=0,\ & \forall \xx\in D_{\sigma}.
\end{align*}
\end{defin}

 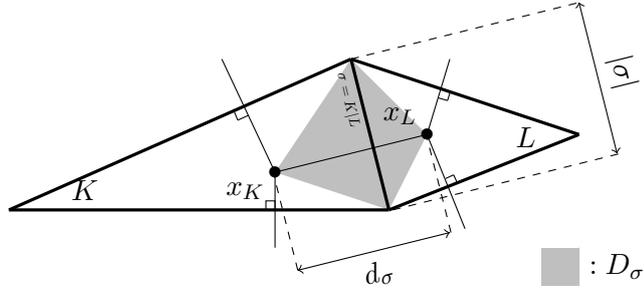
\begin{figure}[!h]
	\begin{center}
		\begin{tikzpicture}
		\fill [color=gray!50] (3.5,0.5) -- (4.5,2) -- (5.5,1) -- (5,0) -- cycle;
		\draw [very thick] (0,0) -- (5,0);
		\draw [very thick] (0,0) -- (4.5,2);
		\draw [very thick] (5,0) -- (4.5,2) node [near end,below,sloped,scale=0.5] {$\sigma=K|L$};
		\draw [very thick] (4.5,2) -- (7.5,1);
		\draw [very thick] (5,0) -- (7.5,1);
		\draw (3.5,0.5) node {$\bullet$};
		\draw (3.5,0.5) node[below left]{$x_K$};
		\draw (5.5,1) node {$\bullet$};
		\draw (5.5,1) node[above left]{$x_L$};
		\draw (3.5,0.5) -- (5.5,1);
		\draw (1,0) node[above]{$K$};
		\draw (6.8,0.7) node[above]{$L$};
		\draw (5.5,1) -- (5.8,2);
		\draw (5.5,1) -- (6,-0.25);
		\draw (3.5,0.5) -- (2.8,2);
		\draw (3.5,0.5) -- (3.5,-0.5);
		\draw (3.5,0) -- (3.5,0.12) -- (3.38,0.12) -- (3.38,0) -- cycle;
		\draw (2.95,1.335) -- (3.01,1.19);
		\draw (3.01,1.19) -- (3.154,1.256);
		\draw (5.8,1.55) -- (5.77,1.45);
		\draw (5.77,1.45) -- (5.65,1.485);
		\draw (5.9,0.356) -- (5.85,0.465);
		\draw (5.85,0.465) -- (5.746,0.417);
		\fill [color=gray!50] (7,-1) -- (7.5,-1) -- (7.5,-0.5) -- (7,-0.5) -- cycle;
		\draw [color=gray!50] (7.5,-1) -- (7.5,-0.5) node [color=black,midway,right] {$:D_{\sigma}$};
	    \begin{scope}[xshift=0.3cm,yshift=-1.3cm]
	    \draw [<->] (3.5,0.5) -- (5.5,1) node[midway,below,sloped] {$d_{\sigma}$};
	    \end{scope}
	    \begin{scope}[xshift=3cm,yshift=0.75cm]
	    \draw [<->] (5,0) -- (4.5,2) node[midway,above,sloped] {$|\sigma|$};
	    \end{scope}
	    \begin{scope}[xshift=0.55cm,yshift=-0.8cm,scale=0.65]
	    \draw [dashed] (5,0) -- (4.5,2);
	    \end{scope}
	    \begin{scope}[xshift=2.565cm,yshift=-0.3cm,scale=0.65]
	    \draw [dashed] (5,0) -- (4.5,2);
	    \end{scope}
	    \begin{scope}[xshift=-0.75cm,yshift=1.25cm,scale=1.5]
	    \draw [dashed] (3.5,0.5) -- (5.5,1);
	    \end{scope}
	    \begin{scope}[xshift=-0.2cm,yshift=-0.75cm,scale=1.5]
	    \draw [dashed] (3.5,0.5) -- (5.5,1);
	    \end{scope}
		\end{tikzpicture}
		\caption{The diamond $D_{\sigma}$}
		\label{fig:D_sigma}
	\end{center}
\end{figure}

Let us then give convergence property of the discrete gradient (see e.g., in the
case of Dirichlet boundary condition, \cite{zbMATH01970883} and
\cite{zbMATH02027732} in $L^2$ context, and \cite{LL12} in the $L^1$ context).

\begin{lemma}[Weak convergence of the finite volume gradient]\label{WCG}
Let $(\M_m)_{m\ge1}$ be a sequence of admissible meshes satisfying \eqref{cmesh}
and such that $h_{\M_m}\rightarrow0$ as $m\rightarrow \infty$. Let $v_{m}\in X(\T_m)$ and let us
assume that there exists $\alpha\in[1,+\infty[$ and $C>0$ such that
$\|v_{m}\|_{1,\alpha,\M_m}\le C$, and that $v_{m}$ converges in
$L^1(\Omega)$ to $v\in W^{1,\alpha}(\Omega)$. Then $\nabla_{\M_m}v_{m}$
converges to $\nabla v$ weakly in $L^{\alpha}(\Omega)^d$.
\end{lemma}

%\begin{thm}[Compactness]\label{compact}
%Let $(\mathcal{T}_m)_{m>1}$ be a sequence of admissible meshes satisfying
%\eqref{cmesh} and such that $h_{\M_m}\to 0$. Let $v_\T\in \mathcal{X}(\M)$,
%such that there exist a $C\in\R$ such that
 %$$\int_{\mathbb{R}^d}|v_\mathcal{T}(x+h)-v_\mathcal{T}(x)|^2\le C |h| \quad
 %\forall |h|<1,$$
   %and
    %$$\int_{\bar{\omega}}|v_\mathcal{T}(x+h)-v_\mathcal{T}(x)|^2\le C |h|(|h|+2h_\mathcal{T}),
    %$$
    %with $\bar{w}$ is a compact in $\Omega$, for $|h|\leq
    %d(\bar{w},\Omega^c)$. Then $v_\T$ is relatively compact in $L^2(\Omega)$.\\
    %Furthermore, if $|v_n|_{1,p,\M}+\|v_n\|_{0,2}$ is bounded, then $v_n$ is
    %relatively compact in $L^2(\Omega)$ and if $v_n \to v$ in $L^2(\Omega)$ as
    %$n\to \infty$, then $v\in H^1(\Omega)$.
     % {\color{red}ce theoreme pas bien écrit et je sais pas si vrai, si tu peux
     % regarder pour savoir lequel écrire}
%\end{thm}%
We now define the finite volume scheme. Let $\M$ be an admissible mesh in the
sense of definition \ref{admi}. %we can define the finite volume discretization
                                %of \eqref{Pb}.
For $K\in \T$ and $\sigma \in \E_{K}$, we define $\vv_{K,\sigma}$ by

\begin{equation}\label{v}
 \vv_{K,\sigma}=\frac{1}{\m(D_{\sigma})}\int_{D_{\sigma}}\vv\cdot \n\, \mathrm{d}x.
\end{equation}
We consider the following finite volume scheme for \eqref{Pb}
\begin{equation}\label{scheme}
\forall K\in\T,\ \sum_{\sigma\in\E_{int}(K)}\frac{\mathrm{m}
  (\sigma)}{d_{\sigma}}\lambda(u)_{\sigma}(u_K-u_L) +
\sum_{\sigma\in\E_{int}(K)}\m(\sigma) \vv_{K,\sigma} u_{\sigma,+}  = \int_K
f\,\mathrm{d}x,
\end{equation}
%with $u_L=0$ if $\sigma\in\E_{ext}\cap\E(K)$
and
\begin{equation}\label{usigma}
  \forall \sigma=K|L\in\E_{int},\   u_{\sigma,+}
  =\left\{ \begin{aligned} & u_K & \textnormal{if}\
      \vv_{K,\sigma}\ge0,
      \\  & u_L \ & \textnormal{otherwise}.
      \end{aligned}\right.% \\
%\forall \sigma\in\E_{ext}\cap\E(K),\ & u_{\sigma,+}=u_K &\ \textnormal{if}\ \vv_{K,\sigma}\ge0,\& & u_{\sigma,+}=0 &\ \textnormal{otherwise}.
\end{equation}
We denote $u_{\sigma,-}$ the downstream choice of $u$, i.e. $u_{\sigma,-}$ is such that $\{u_{\sigma,+},u_{\sigma,-}\}=\{u_K,u_L\}$, $\forall\sigma\in \E_{int}$.

%(with $u_L=0$ if $\sigma\in\E_{ext}\cap\E(K)$).\\
Finally,% for $\sigma \in \E$
\begin{align}\label{usigma2}
\forall \sigma=K|L\in\E_{int},\ \min[\lambda(u_K),\lambda(u_L)] \le \lambda(u)_{\sigma}\le \max[\lambda(u_K),\lambda(u_L)],
%&\label{usigma3}\forall \sigma\in\E_{ext}\cap\E(K),\ \lambda(u)_{\sigma}=\lambda(u_K),
\end{align}
where $\lambda(u)_{\sigma}$ is for example the mean value of $\lambda(u_K)$ and $\lambda(u_L)$ if $\sigma\in\E_{int}$.

\section{Main results}
%Let us state our main results, which we shall prove in the following sections.

Our main results on the finite volume scheme are the following.  The first one states that there exists at least one solution to the scheme. It is a generalization of Theorem 2.5 in \cite{CHD11} in the context of a quasilinear problem with a median value constraint instead of a mean value constraint. The second one gives the convergence of this solution to the unique  renormalized solution  of the continuous problem with null median, as the size of the mesh tends to 0.
\begin{thm}[Existence of the solution of the scheme]\label{exist}
Let us assume that  \eqref{assump1}--\eqref{assump2} hold. Let $\M$ be an
admissible mesh in the sense of Definition \ref{admi} satisfying
\eqref{cmesh}. Then there exists a solution $u_{\T}=(u_K)_{K\in\M}$ to
\eqref{scheme}--\eqref{usigma2} having $\underline{\med}(u_\T)=0$.
\end{thm}{}

\begin{thm}[Convergence of the solution of the scheme]\label{conv}Let
  $(\M_m)_{m\geq1}$ be a sequence of admissible meshes in the sense of
  Definition \ref{admi}, which satisfy \eqref{cmesh} and such that $h_{\M_m}$
  goes to $0$ as $m\to\infty$. Let $u_{m}=(u_K^{m})_{K\in\T_m}\in X(\T_{m})$ be a solution of
  \eqref{scheme} such that $\underline{\med}(u_m)=0$.  Then $u_m$ converges
 to  the unique renormalized solution $u$ of \eqref{Pb} having
  $\med(u)=0$, in the sense that
  \begin{gather*}
    \text{$u_m$ converges to $u$ a.e. in $\Omega$,} \\
    \forall n\in\N, \ \nabla_{\M_m}T_n(u_m) \text{ converges
      to $\nabla T_{n}(u)$, weakly in }
       (L^{2}(\Omega))^{d},
\end{gather*}
as $m\rightarrow \infty$.
\end{thm}

 \begin{remark}
As explained in Introduction we choose in the present paper a constraint on the
median value instead of the mean value. It allows one to mix the techniques
developed in \cite{BGM1} and the finite volume. Observe that the median is an
appropriate choice in \cite{ACMM,BGM1} to deal with nonlinear elliptic equations
with $L^1$ data and Neumann boundary conditions since we cannot expect to have a
solution $u$ (in the sense of distribution or in the renormalized sense) of
$-\Delta_p u=f$ with $p$ closed to $1$ such that the solution belongs to
$L^1(\Omega)$. However under the restriction
$p>2-1/N$ and using the Boccardo-Gallou\"et estimates it is possible to solve
$-\Delta_p u=f$ with $f$ in $L^1$ and Neumann boundary conditions in the sense
of distributions with a mean value equal to zero, see \cite{Prignet97}.
As far as equation \eqref{Pb} is concerned a natural question is to solve its
and to approximate its with $\int_\Omega u dx=0$ and not $\med(u)=0$. To our
knowledge the continuous case is not dealt in the literature. Starting from an
approximate problem the difficulties are similar in passing to the limit in the
continuous case and in the discrete case : the crucial steps are the a priori
estimates stated  in Section 4.  Since we cannot give all the details of a possible proof
we refer to \cite{mirella}.
% Propositions \ref{estlog}, \ref{EstTn} and
% \ref{propestdiscrete}.
% Since we cannot give all the details of a possible proof in the present paper we
% state in the following the principal arguments.

\end{remark}

 \section{A priori estimates}

%{\color{red} AJOUTER UNE PHRASE explications}

This section is devoted to derive {\em a priori} estimates of the solution of
the scheme \eqref{scheme}, which are crucial to extract subsequences using
compactness results and then to pass to the limit in the scheme. Let us observe
that we adapt the strategy developed in \cite{BGM1} for the continuous problem
\eqref{Pb1} with Neumann boundary conditions to the discrete case and that we
use the techniques developed in \cite{CHD11} for the approximation of the
solution to problem \eqref{Pb1} with a more regular data and Neumann boundary
conditions and the ones of \cite{DGH03} and \cite{leclavier} which study the
approximation of equations with $L^1$ (or measure data) with Dirichlet boundary
conditions.

\begin{prop}[Estimate on $\ln(1+|u_\M|)$ with $\underline{\med}(u_\M)=0$]\label{estlog}
Let $\M$ be an admissible mesh satisfying \eqref{cmesh}. If
$u_{\M}=(u_K)_{K\in\T}$ is a solution to (\ref{scheme}), such that $\underline{\med}(u_\M)=0$, then
\begin{equation} \label{estlogeq}
\| \ln (1 + |u_{\M}|)\|_{1,2,\M}^2 \le C\left( 2\|f\|_{L^1(\Omega)} + d
  |\Omega|^{\frac{p-2}{p}}\, \|\,\vv\,\|^2_{\left(L^p(\Omega)\right)^d}\right),
\end{equation}
where 
%$|\vv|$ denotes the Euclidean norm of $\vv$ in $\R^d$ and 
$C=C(\Omega,\mu, \beta, p,\xi)$ is a positive constant.
\end{prop}

\begin{proof}
A $log$-estimate was obtained in [Proposition 3.1, \cite{DGH03}] in the case of Dirichlet boundary conditions and $\vv \in (C(\bar{\Omega}))^d$. Since we deal with Neumann boundary conditions, $\vv \in (L^p(\Omega))^d$ and the specific choice of $\underline{\med}(u_\M)=0$, we will adapt the proof derived in \cite{DGH03} and explain the modifications.
%By Adapting the proof of the $log$-estimate in \cite{DGH03}, for $f\in L^1(\Omega)$ and $\vv\in L^p(\Omega)^d$, we get
As in \cite{DGH03}, let $\varphi(s)=\displaystyle \int_0^s\dfrac{dt}{(1+|t|)^2}$. Taking $\varphi(u_K)$ as a test function in the scheme \eqref{scheme} and reordering the sums yield
\begin{equation}\label{premiere}
    \sum_{\sigma\in\E_{int}}\frac{\m(\sigma)}{d_\sigma}\lambda(u)_\sigma(u_K-u_L)(\varphi(u_K)-\varphi(u_L))\leq \|f\|_{L^1(\Omega)}+\sum_{\sigma\in\E_{int}}\m(\sigma)|\vv_{K,\sigma}|u_{\sigma,+}(\varphi(u_{\sigma,-})-\varphi(u_{\sigma,+})).
\end{equation}
To control the second term of the right-hand side of \eqref{premiere} we introduce the set of edges $\A$ (see \cite{DGH03}) by
\begin{equation}\label{SA}
\A=\{\sigma\in\E_{int} \tq u_{\sigma,+}\ge u_{\sigma,-},\ u_{\sigma,+}<0\}\cup\{\sigma\in\E_{int} \tq u_{\sigma,+}< u_{\sigma,-},\ u_{\sigma,+}\ge 0\},
\end{equation}
Since $\varphi$ is non-decreasing, as in \cite{DGH03} we obtain
\begin{equation}
    \sum_{\sigma\in\E_{int}}\m(\sigma)|\vv_{K,\sigma}|u_{\sigma,+}(\varphi(u_{\sigma,-})-\varphi(u_{\sigma,+})) \leq \sum_{\sigma\in\A}\m(\sigma)|\vv_{K,\sigma}|u_{\sigma,+}(\varphi(u_{\sigma,-})-\varphi(u_{\sigma,+})).
\end{equation}
Now using  Cauchy-Schwarz and Hölder inequalities, and the following inequality (see Lemma 3.1, \cite{DGH03}), $\forall \sigma \in \A, \, |u_{\sigma,+}|^2|\varphi(u_{\sigma,-})-\varphi(u_{\sigma,+})|^2 \leq |u_{\sigma,-} - u_{\sigma,+}||\varphi(u_{\sigma,-})-\varphi(u_{\sigma,+})|$, we obtain
\begin{align}
   \sum_{\sigma\in\A}\m(\sigma)|\vv_{K,\sigma}||u_{\sigma,+}|(\varphi(u_{\sigma,-})-\varphi(u_{\sigma,+})) \leq\nonumber
   & \left(\sum_{\sigma\in\A}\m(\sigma)d_\sigma|\vv_{K,\sigma}|^2\right)^{\frac{1}{2}}\nonumber\\
   &\times \left(\sum_{\sigma \in \A}\frac{\m(\sigma)}{d_\sigma}|u_{\sigma,+}|^2|\varphi(u_{\sigma,-})-\varphi(u_{\sigma,+})|^2\right)^{\frac{1}{2}}\nonumber\\\leq
    &\left(\sum_{\sigma\in\A}\m(\sigma)d_\sigma\right)^{\frac{p-2}{p}}\left(\sum_{\sigma\in\A}\m(\sigma) d_\sigma|\vv_{K,\sigma}|^p\right)^{\frac{1}{p}}\nonumber\\
    &\times\left(\sum_{\sigma \in \A}\frac{\m(\sigma)}{d_\sigma}|u_{\sigma,+}|^2|\varphi(u_{\sigma,-})-\varphi(u_{\sigma,+})|^2\right)^{\frac{1}{2}} \nonumber
    \\ \leq
    &\left(\sum_{\sigma\in\A}\m(\sigma)d_\sigma\right)^{\frac{p-2}{p}}\left(\sum_{\sigma\in\A}\m(\sigma) d_\sigma|\vv_{K,\sigma}|^p\right)^{\frac{1}{p}}\nonumber\\
    &\times\left(\sum_{\sigma \in \A}\frac{\m(\sigma)}{d_\sigma}|u_k - u_{L}||\varphi(u_{K})-\varphi(u_{L})|\right)^{\frac{1}{2}}.
    \end{align}

Recalling that  $\sum_{\sigma\in\A}\m(\sigma)d_\sigma\leq\sum_{\sigma\in\E_{int}}\m(\sigma)d_\sigma=d\m(\Omega)$ and since the term $\left(\sum_{\sigma\in\E_{int}}\m(\sigma)d_\sigma|\vv_{K,\sigma}|^p\right)^\frac{1}{p}$ is bounded by $d^{\frac{1}{p}}\|\vv\|_{(L^p(\Omega))^d}$,  by Young's inequality  we get
\begin{equation}
\begin{aligned}
    \sum_{\sigma\in\A}\m(\sigma)|\vv_{K,\sigma}||u_{\sigma,+}|(\varphi(u_{\sigma,-})-\varphi(u_{\sigma,+})) &\leq \frac{1}{2\beta}d\m(\Omega)^{\frac{p-2}{p}}\|\vv\|_{(L^p(\Omega))^d}^2
    \\&+\frac{\beta}{2}\m(\Omega)^{\frac{p-2}{p}}\sum_{\sigma\in\E_{int}}\frac{\m(\sigma)}{d_\sigma}(u_K-u_L)(\varphi(u_K)-\varphi(u_L)),
    \end{aligned}
\end{equation}
where $\beta>0$. Since $0<\mu\leq\lambda(u)$, an appropriate choice of $\beta$ gives
\begin{equation}
    \sum_{\sigma\in\E_{int}}\frac{\m(\sigma)}{d_\sigma}(u_K-u_L)(\varphi(u_K)-\varphi(u_L)) \leq C(\Omega, \mu, \beta, p) \left(2\|f\|_{L^1(\Omega)} + d|\Omega|^{\frac{p-2}{p}}\|\vv\|^2_{(L^p(\Omega))^d}\right).
\end{equation}
Moreover we have, for all $(x,y)\in \mathbb{R}^2$, $\left(\ln(1+|x|)-\ln(1+|y|) \right)^2 \leq (x-y)(\varphi(x)-\varphi(y))$.
It follows that
\begin{equation}
    \sum_{\sigma\in\E_{int}}\frac{\m(\sigma)}{d_\sigma}\left(\ln(1+|u_K|)-\ln(1+|u_L|)\right)^2 \leq C(\Omega, \mu, \beta, p) \left(2\|f\|_{L^1(\Omega)} + d|\Omega|^{\frac{p-2}{p}}\|\vv\|^2_{(L^p(\Omega))^d}\right).
\end{equation}
Since $\underline{\med}(\ln(1+u_\M))=0$, the discrete Poincaré-Wirtinger inequality \eqref{PW} implies that
\begin{equation*}
\| \ln (1 + |u_{\M}|)\|_{1,2,\M}^2 \le C\left( 2\|f\|_{L^1(\Omega)} + d |\Omega|^{\frac{p-2}{p}}\, \|\,\vv\,\|^2_{\left(L^p(\Omega)\right)^d}\right).
\end{equation*}
\end{proof}{}

Let us state a corollary which is a consequence of Proposition \ref{estlog} and is
necessary for the proof of the estimate of Proposition \ref{EstTn} and for Proposition \ref{propestdiscrete}. It may be found in \cite{DGH03} and is recalled here with its proof, for the sake of completeness.

\begin{corollary}\label{log}
Let $\M$ be an admissible mesh satisfying \eqref{cmesh}. If
$u_{\M}=(u_K)_{K\in\T}\in X(\T)$ is a solution to
(\ref{scheme}) and, for $n>0$, $E_n=\{|u_{\M}|>n\}$, then there exists $C>0$
only depending on $(\Omega,\vv,f,d,p,\xi)$ such that
\begin{equation}\label{measEn}
\meas(E_n)\le \frac{C}{(\ln (1+n))^2}.
\end{equation}
\end{corollary}

\begin{proof}
On the one hand, using Proposition \ref{estlog} we have
\begin{equation}\label{first}
    \| \ln (1 + |u_{\M}|)\|_{1,2,\M}^2 \le C \left( 2\|f\|_{L^1(\Omega)} + d |\Omega|^{\frac{p-2}{p}}\, \|\,\vv\,\|^2_{(L^p(\Omega))^d}\right).
    \end{equation}
On the other hand, since $\underline{\med}(\ln(1+|u_\M|))=0$, by the discrete Poincaré-Wirtinger median inequality \eqref{PW}, we have that there exists $C>0$ only depending on $(\Omega, d,  \xi)$ such that
\begin{equation}\label{second}
    \|\ln(1+|u_\M|)\|_{0,2}\leq C|u_{\M}|_{1,2,\M}.
\end{equation}{}
Therefore, using \eqref{first} and \eqref{second}, there exists $C>0$ only depending on $(\Omega,\vv,f,d,p,\xi)$ such that
\begin{equation} \label{third}
    \|\ln(1+|u_\M|)\|_{0,2}^2\leq C.
\end{equation}
Finally, due to the fact that $\meas(E_n)=\meas\left(\{\ln(1+|u_\M|)\geq
  \ln(1+n)\}\right)$, the Chebyshev inequality  and \eqref{third} lead to the
result.
\end{proof}{}

\begin{prop}[Estimate on $T_n(u_{\M})$]\label{EstTn}
Let $\M$ be an admissible mesh satisfying \eqref{cmesh}. If
$u_{\M}=(u_K)_{K\in\T}\in X(\T)$ is a solution to
(\ref{scheme}) having $\underline{\med}(u_\M)=0$, then for any $n\geq 0$, there exists $C>0$ only
depending on $(\Omega,\vv, f, n, d,\xi)$ such that
\begin{equation}\label{estTn}
\|T_n(u_{\M})\|_{1,2,\M} \le C.
\end{equation}
Let  $(\M_m)_{m\geq 1}$ be a sequence of admissible meshes satisfying
\eqref{cmesh} and such that $h_{\M_{m}}$ goes to zero as $m\rightarrow \infty$
and let $u_{m}=(u_K^{m})_{K\in\T_{m}}\in X(\T_{m})$ be a solution to
(\ref{scheme}) having $\underline{\med}(u_m)=0$.
Then there exists a measurable function $u$ finite a.e. in $\Omega$
such that, up to a subsequence (still indexed by $m$),
\begin{gather}
    T_n(u)\in H^1(\Omega),\ \textnormal{for any }n>0,  \label{i}
    \\
     \med(u)=0, \label{4.14bis}\\
     T_n(u_{m}) \rightarrow T_n(u)\ \textnormal{strongly in }L^2(\Omega)\ \textnormal{and a.e },\label{j}\\
     \nabla_{\mathcal{M}_{m}}T_n(u_{m})\rightharpoonup \nabla T_n(u)\
     \textnormal{in }(L^2(\Omega))^d, \,\, \forall n>0\label{k}.
\end{gather}
%$T_n(u_{m})$ converges to $T_n(u)$ strongly in $L^2(\Omega)$ and a.e. in
%$\Omega$, and $\nabla_\mathcal{M}T_n(u_{\mathcal{T}_m})\rightharpoonup \nabla
%T_n(u)$ in $(L^2(\Omega))^d$.
\end{prop}

\begin{proof}
The proof is divided into 2 steps. First, we prove that $T_n(u_\M)$ satisfies
the a priori estimate \eqref{estTn}. In the second step, considering a sequence
of admissible meshes ${\mathcal M}_{m}$,  we prove that the solution $u_{m}$ to the scheme
\eqref{scheme} converges to a function $u$ as $m$ goes to infinity and that
\eqref{i}--\eqref{k} hold true.
% $T_n(u)$ lies in $H^1(\Omega)$, that $\med(u)=0$ and
% $\nabla_{\mathcal{M}_{m}}T_n(u_{m})$ converge weakly to $\nabla T_n(u)$ in
% $(L^{2}(\Omega))^{d}$ as $m$ goes to infinity.

\smallskip

\noindent{\bf {Step 1.} \it{Estimate on $T_n(u_\M)$.}} \\
%We prove that $$\|T_n(u_{\M})\|_{1,2,\M} \le C\ \ \ \ \forall n>0.$$
After multiplying each equation of the scheme by $T_n(u_K)$,  summing over each control volume and reordering the sums, we obtain $T_1+T_2=T_3$ with
\begin{align*}
T_1 &= \sum_{\sigma\in\E_{int}} \frac{\m(\sigma)}{d_{\sigma}}\lambda(u)_{\sigma} (u_K-u_L) (T_n(u_K)-T_n(u_L)), \\
T_2 &= \sum_{\sigma\in\E_{int}} \m(\sigma) \vv_{K,\sigma} u_{\sigma,+} (T_n(u_K)-T_n(u_L)), \\
T_3 &= \sum_{K\in\T} \int_K f\,T_n(u_K)\,\mathrm{d}x.
\end{align*}
 Since $T_n$ is bounded by $n$, we obtain that $ |T_3| \le n \|f\|_{L^1(\Omega)}$. Then
\begin{equation}
    T_1\leq n\|f\|_{L^1(\Omega)}-T_2.
\end{equation}
Let $\sigma\in\E$. By the definition (\ref{usigma}) of $u_{\sigma,+}$ and recalling that $u_{\sigma,-}$ is the downstream choice of $u$, if $\vv_{K,\sigma}\ge 0$ it gives
\begin{equation*}
\vv_{K,\sigma}(T_n(u_K)-T_n(u_L)) = \vv_{K,\sigma}(T_n(u_{\sigma,+})-T_n(u_{\sigma,-})),
\end{equation*}
and if $\vv_{K,\sigma}<0$ it gives
\begin{equation*}
\vv_{K,\sigma}(T_n(u_K)-T_n(u_L)) = -\vv_{K,\sigma}(T_n(u_{\sigma,+})-T_n(u_{\sigma,-})).
\end{equation*}
In consequence, $T_2$ can be written as

\begin{equation}\label{t2}
-T_2 = \frac{1}{n}\sum_{\sigma\in\E_{int}} \m(\sigma)\,|\vv_{K,\sigma}|\,u_{\sigma,+}(T_n(u_{\sigma,-})-T_n(u_{\sigma,+})).
\end{equation}
As in the proof of estimate \eqref{estlogeq}, we use the set of edges
$\A=\{\sigma\in\E_{int} \tq u_{\sigma,+}\ge u_{\sigma,-},\ u_{\sigma,+}<0\}\cup\{\sigma\in\E_{int} \tq u_{\sigma,+}< u_{\sigma,-},\ u_{\sigma,+}\ge 0\}$.
Since $T_n$ is non decreasing we have
\begin{equation}\label{-T2}
-T_2 \le \frac{1}{n}\sum_{\sigma\in\A} \m(\sigma)\,|\vv_{K,\sigma}|\,u_{\sigma,+}(T_n(u_{\sigma,-})-T_n(u_{\sigma,+})).
\end{equation}
Due to the fact that $\forall\sigma\in\A$, if $|u_{\sigma,+}|\ge n$ then $|u_{\sigma,-}|\ge n$ we deduce that
\begin{equation*}
 \ u_{\sigma,+} (T_n(u_{\sigma,-})-T_n(u_{\sigma,+}))= T_n(u_{\sigma,+}) (T_n(u_{\sigma,-})-T_n(u_{\sigma,+})),\,\, \forall \sigma \in \A.
\end{equation*}
It follows that
\begin{align*}
    -T_2  \le \sum_{\sigma\in\A} \m(\sigma)\, |\vv_{K,\sigma}|\, T_n(u_{\sigma,+})\, \left(T_n(u_{\sigma,-}) - T_n(u_{\sigma,+}) \right).
\end{align*}
Moreover, using Cauchy-Schwarz and Young inequalities, taking into account that\\$\left( \sum\limits_{\sigma\in\E_{int}} \m(\sigma)d_\sigma|\vv_{K,\sigma}|^2\right)^\frac{1}{2}$ is bounded by $d^{\frac{1}{2}}\|\vv\|_{(L^2(\Omega))^d}$, we obtain
\begin{align*}
-T_2 & \le \sum_{\sigma\in\A} \m(\sigma)\, |\vv_{K,\sigma}|\, T_n(u_{\sigma,+})\, \left(T_n(u_{\sigma,-}) - T_n(u_{\sigma,+}) \right) \\
     & \le \left( \sum_{\sigma\in\E_{int}} \m(\sigma) d_{\sigma} |\vv_{K,\sigma}|^2 \right)^{\frac{1}{2}}  \left( \sum_{\sigma\in\A}\frac{\m(\sigma)}{d_{\sigma}} T_n(u_{\sigma,+})^2 \left(T_n(u_{\sigma,-}) -T_n(u_{\sigma,+}) \right)^2 \right)^{\frac{1}{2}} \\
     & \le n\,d^{\frac{1}{2}} \|\vv\|_{L^2(\Omega)^d} \left( \sum_{\sigma\in\A}\frac{\m(\sigma)}{d_{\sigma}}  \left( T_n(u_{\sigma,-}) - T_n(u_{\sigma,+}) \right)^2 \right)^{\frac{1}{2}}\\
     & \le \frac{1}{2\beta} n^2\,d \|\vv\|_{L^2(\Omega)^d}^2 + \frac{\beta}{2} \sum_{\sigma\in\A}\frac{\m(\sigma)}{d_{\sigma}} (T_n(u_{\sigma,-})-T_n(u_{\sigma,+}))^2  \\
     & \le \frac{1}{2\beta} n^2\,d \|\vv\|_{L^2(\Omega)^d}^2 + \frac{\beta}{2} \sum_{\sigma\in\E_{int}}\frac{\m(\sigma)}{d_{\sigma}} (T_n(u_K)-T_n(u_L))^2,
\end{align*}
where $\beta>0$. Since $T_n(u_K)-T_n(u_L) \le u_K-u_L$, we have
\begin{equation*}
-T_2 \le \frac{1}{2\beta} n^2\,d \|\vv\|_{(L^2(\Omega))^d}^2 + \frac{\beta}{2} \sum_{\sigma\in\E_{int}}\frac{\m(\sigma)}{d_{\sigma}} (u_K-u_L) \left(T_n(u_K) - T_n(u_L) \right),
\end{equation*}
and we can deduce that
\begin{equation*}
\begin{aligned}
 \sum_{\sigma\in\E_{int}}\lambda(u)_\sigma\frac{\m(\sigma)}{d_{\sigma}} (u_K-u_L) \left(T_n(u_K) - T_n(u_L) \right) &\le n \|f\|_{L^1(\Omega)} + \frac{1}{2\beta} n^2\,d \|\vv\|_{(L^2(\Omega))^d}^2
 \\&+\frac{\beta}{2} \sum_{\sigma\in\E_{int}}\frac{\m(\sigma)}{d_{\sigma}} (u_K-u_L) \left(T_n(u_K) - T_n(u_L) \right).
 \end{aligned}
\end{equation*}
Recalling that $\underline{\med}(T_n(u_\M))=0$ and $0\leq\mu\leq\lambda(u)$, an appropriate choice of $\beta$ and the Poincaré-Wirtinger inequality \eqref{PW} lead to the result.

%Therefore, using again the fact that $|T_n(u_K)-T_n(u_L)| \le |u_K-u_L|$ yields
%\begin{equation*}
%\frac{1}{2} \sum_{\sigma\in\E}\frac{|\sigma|}{d_{\sigma}} (T_n(u_K)-T_n(u_L))^2 \le n \|f\|_{L^1(\Omega)} + \frac{1}{2} n^2\,d \|\vv\|_{L^2(\Omega)^d}^2.
%\end{equation*}

\smallskip

\noindent{\bf{Step 2.} \it{In this step we consider
    sequence of admissible meshes  $(\M_m)_{m\geq 1}$ satisfying \eqref{cmesh} and such that
    $h_{\M_{m}}$ goes to zero as $m\rightarrow \infty$. If
    $u_{m}=(u_K^{m})_{K\in\T_{m}}\in X(\T_{m})$ denotes  a solution to
(\ref{scheme}) having $\underline{\med}(u_m)=0$, we show
that there exists a measurable
    function $u$ finite a.e. in $\Omega$  such that \eqref{i}--\eqref{k} hold true.}}

The method is widely used for elliptic equations with $L^1$ (or measure data)
(see e.g. \cite{zbMATH01463265}) and consists in proving that, up to
subsequence, $u_{m}$ is a Cauchy sequence in measure. For the convenience of
the reader we give the complete arguments. For any $n$, in view of Step 1 we
know that the sequence $\big(\|T_{n}(u_{m})\|_{1,2,\mathcal{M}_{m}})\big)_{m\geq 1}$ is bounded
(uniformly with respect to $m$). By Theorem~\ref{compact} and a diagonal process
($n$ being a natural number), up to a subsequence still indexed by $m$, we
deduce that, for any $n\in\N$, there exists $v_{n}$ belonging to $H^{1}(\Omega)$ such that

\begin{equation}
  T_{n}(u_{m}) \rightarrow v_{n}, \text{ a.e. in $\Omega$, as $m\rightarrow
    \infty$}.\label{eq:og0}
\end{equation}

We now prove that $u_{m}$ is a Cauchy sequence in measure. Let $\omega>0$. For
all $n>0$, and  all
$m,p\geq 0$ , we have
\begin{equation*}
\{ |u_{m}-u_{p}|>\omega\} \subset \{ |u_{m}|>n\} \cup \{ |u_{p}|>n\} \cup \{ |T_n(u_{m})-T_n(u_{p})|>\omega\}.
\end{equation*}
Let $\varepsilon>0$ fixed. By Corollary \ref{log}, let $n>0$ such that, for all
$m,p\geq 0$,
\begin{equation*}
\textnormal{meas}(\{ |u_{m}|>n\})+\textnormal{meas}(\{ |u_{p}|>n\}) < \frac{\varepsilon}{2}.
\end{equation*}
Once $n$ is chosen, since $T_{n}(u_{m})$ converges almost everywhere to $v_{n}$
as $m$ goes to infinity we obtain
\begin{equation*}
\exists m_0>0 \tq \forall m,p\geq m_{0} \quad \textnormal{meas}(\{ |T_n(u_{m})-T_n(u_{p})|>\omega\}) \le \frac{\varepsilon}{2}.
\end{equation*}
Therefore, we deduce that $\forall m,p\geq m_{0}$
\begin{equation*}
\textnormal{meas}\{ |u_{m}-u_{p}|>\omega\} < \varepsilon.
\end{equation*}
Hence $(u_{m})_{m\in\N}$ is a Cauchy sequence in measure. Consequently, up to a subsequence still indexed by $m$, there exists a measurable function $u$ such that
\begin{equation}\label{conver}
    u_{m}\rightarrow u\,\text{ a.e. in $\Omega$. }
  \end{equation}
  It follows from Corollary \ref{log} that $u$ is finite a.e. in
  $\Omega$.

  Moreover by the pointwise convergence \eqref{eq:og0} of
  $T_{n}(u_{m})$ for any $n\in\N$ we deduce that $T_{n}(u)=v_{n}\in H^{1}(\Omega)$.
% \smallskip
% \noindent{\bf{Step 3.} \it{We now to prove that \eqref{i}, \eqref{j} and \eqref{k}} hold.} \\
% On one hand we deduce from Step 1 that $\|T_n(u_M)\|_{0,2,\M}$ is bounded for
% any $n\geq 1$. As a consequence of Theorem \ref{compact}, there exists a
% subsequence still denoted by $\M_m$ and a measurable function $v_n$ belonging to
% $H^1(\Omega)$ such that $T_n(u_\M)$ converges to $v_n$, for any $n\geq 1$.
% On the other hand, from Step 2 we deduce that $u_\M$ converges to $u$ a.e.. Therefore, we get that
% \begin{equation*}
%     T_n(u)\in H^1(\Omega).
% \end{equation*}
Applying Theorem \ref{WCG} we obtain that
\begin{equation*}
    \nabla_{\M_{m}}T_n(u_m)\rightharpoonup \nabla T_n(u)\ \textnormal{in}\
    (L^2(\Omega))^d, \text{ as $m\rightarrow\infty$.}
\end{equation*}

It remains to prove that $\med(u)=0$. Due to the point-wise convergence of
$u_{_m}$ to $u$, the sequence $\mathds{1}_{\{u_{m}>0\}}\mathds{1}_{\{u>0\}}$
converges to $\mathds{1}_{\{u>0\}}$ a.e. as $h_{\M}$ goes to zero. Recalling
that $\underline{\med}(u_m)=0$ Fatou's lemma leads to
\begin{align*}
    \meas\{u(x)>0\}&\leq \liminf{\int_\Omega \mathds{1}_{\{u_m>0\}}\mathds{1}_{\{u>0\}}}\mathrm{d}x\\
    &\leq \liminf \meas\{u_m(x)>0\}\\
    &\leq \dfrac{\meas(\Omega)}{2}.
\end{align*}
Analogously from the convergence of $\mathds{1}_{\{u_{m}<0\}}\mathds{1}_{\{u<0\}}$ to $\mathds{1}_{\{u<0\}}$ a.e. as $m\rightarrow\infty$
\begin{align*}
    \meas\{u(x)<0\}\leq \frac{\meas(\Omega)}{2}.
\end{align*}
It follows that $0 \in \med(u)$. Since we have  for $n$ large enough
$\med(T_n(u))=\med(u)$ and since $T_n(u)$ belongs to $H^1(\Omega)$, the median of $u$
is unique and it is equal to $0$.
\end{proof}

Let us recall that in the renormalized framework the decay of the energy
\eqref{def2} plays an important role to derive stability or uniqueness
results. In the following proposition we show a discrete version of the decay of
the energy (uniformly with respect to the sequence of the admissible
meshes). Having \eqref{energiediscrete} ans \eqref{corenergiediscrete} is
crucial to pass to the limit in the scheme.

\begin{prop}[Discrete estimate on the energy]\label{propestdiscrete}
Let $(\M_m)_{m\ge1}$ be a sequence of admissible meshes satisfying
\eqref{cmesh} and such that $h_{\M_{m}}\rightarrow 0$ as $m\rightarrow
\infty$. For any $m\geq 0$, let us consider  $u_{m}=(u_K^{m})_{K\in\T_m}\in X(\T_{m})$  a
solution to (\ref{scheme}) and let $u$ be a measurable function finite a.e. in
$\Omega$ such that, up to a subsequence still indexed by $m$, the second part of Proposition \ref{EstTn} holds.
Then we have
\begin{equation}\label{energiediscrete}
\lim_{n\rightarrow+\infty}\varlimsup_{h_{\M_m}\rightarrow0}
\frac{1}{n}\sum_{\sigma\in\E_{int}}\frac{\m(\sigma)}{d_{\sigma}}\lambda(u_{m})_{\sigma}(u^{m}_K-u^{m}_L)(T_n(u^{m}_K)-T_n(u^{m}_L))=0,
\end{equation}
where $u_L=0$ if $\sigma\in\E_{ext}$, and
\begin{equation}\label{corenergiediscrete}
\lim_{n\rightarrow+\infty}\varlimsup_{h_{\M_m}\rightarrow 0} \frac{1}{n}
\sum_{\sigma\in\E_{int}} \m(\sigma)\, |\vv_{K,\sigma}|\, |u^{m}_{\sigma,+}|\,
|T_n(u^{m}_{\sigma,+})-T_n(u^{m}_{\sigma,-})| = 0.
\end{equation}
\end{prop}

\begin{proof}
Let $m\geq 1$ and $u_m=\big(\um_{K}\big)_{K\in\T_{m}}$ be a solution of \eqref{scheme}. Multiplying each
equation of the scheme by $\frac{T_n(\um_K)}{n}$, summing over $K\in\M$ and
gathering by edges we find
\[
  T_1+T_2=T_3
\]
with
\begin{align}
T_1 &=
      \frac{1}{n}\sum_{\sigma\in\E_{int}}\frac{\m(\sigma)}{d_{\sigma}}\lambda(u_{m})_\sigma(\um_K-\um_L)(T_n(\um_K)-T_n(\um_L)), \label{T1}
  \\
T_2 &= \frac{1}{n}\sum_{\sigma\in\E_{int}} \m(\sigma) \vv_{K,\sigma}\um_{\sigma,+} (T_n(\um_K)-T_n(\um_L)), \label{T2} \\
T_3 &= \frac{1}{n}\sum_{K\in\M} \int_K f\,T_n(\um_K)\,\mathrm{d}x. \label{T3}
\end{align}
According to the definition of $u_{m}$, we have
\begin{equation*}
T_3 = \int_{\Omega}f\frac{T_n(u_m)}{n}\,\mathrm{d}x.
\end{equation*}
Due to \eqref{conver}, $T_n(u_m)$ converges to $T_n(u)$ as $m$ goes to infinity in $L^\infty(\Omega)$ weak-$\star$ and a.e. Since $f$ belongs to $L^1(\Omega)$ it follows that
\begin{equation*}
    \lim_{h_{\M_{m}}\rightarrow0}T_3=\int_{\Omega}f\frac{T_n(u)}{n}\,\mathrm{d}x.
\end{equation*}
Recalling that $u$ is finite almost everywhere in $\Omega$, $\frac{T_n(u)}{n}$
converges to $0$ a.e. in $\Omega$ as $n$ goes to infinity. Therefore, since $f$
belongs to $L^1(\Omega)$ and $\left|\dfrac{T_n}{n}\right|$ is bounded by one,
the Lebesgue dominated theorem allows one to conclude that
\begin{equation}\label{limT3}
\lim_{n\rightarrow+\infty}\lim_{h_{\M_{m}}\rightarrow0}T_3=0.
\end{equation}
We now study the term $T_2$. We know from \eqref{t2} that it can be written as
\begin{equation}\label{+T2}
T_2 = \frac{1}{n}\sum_{\sigma\in\E_{int}}
\m(\sigma)\,|\vv_{K,\sigma}|\,\um_{\sigma,+}(T_n(\um_{\sigma,+})-T_n(\um_{\sigma,-})).
\end{equation}
 Recalling the definition of the subset of edges $\A$
\begin{equation}\label{EA}
\A=\{\sigma\in\E_{int} \tq \um_{\sigma,+}\ge \um_{\sigma,-},\ \um_{\sigma,+}<0\}\cup\{\sigma\in\E_{int} \tq \um_{\sigma,+}< \um_{\sigma,-},\ \um_{\sigma,+}\ge 0\},
\end{equation}
we denote by $\B$ the subset of edges such that
\begin{equation}\label{EB}
\B=\{\sigma\in\E_{int} \tq \um_{\sigma,+}\ge \um_{\sigma,-},\ \um_{\sigma,+}\ge 0\}\cup\{\sigma\in\E_{int} \tq \um_{\sigma,+}< \um_{\sigma,-},\ \um_{\sigma,+}< 0\}.
\end{equation}{}
Then $T_2$ can be written as
\begin{equation}
\begin{aligned}
 T_2= & \frac{1}{n}\sum_{\sigma\in\A}
 \m(\sigma)\,|\vv_{K,\sigma}|\,\um_{\sigma,+}(T_n(\um_{\sigma,+})-T_n(\um_{\sigma,-}))
 \\ & + \frac{1}{n}\sum_{\sigma\in\B}
 \m(\sigma)\,|\vv_{K,\sigma}|\,\um_{\sigma,+}(T_n(\um_{\sigma,+})-T_n(\um_{\sigma,-}))
 \\ & =T_{2,1}+T_{2,2}.
\end{aligned}
\end{equation}
Using the fact that $T_n$ is non decreasing and $T_n(0)=0$, we notice that
\begin{align*}
|T_{2,1}| &= -T_{2,1} = -\frac{1}{n} \sum_{\sigma\in\A} \m(\sigma)\,
            |\vv_{K,\sigma}|\, \um_{\sigma,+} (T_n( \um_{\sigma,+})-T_n(
            \um_{\sigma,-})), \\
|T_{2,2}| &= T_{2,2} = \frac{1}{n} \sum_{\sigma\in\B} \m(\sigma)\,
            |\vv_{K,\sigma}|\, \um_{\sigma,+} (T_n( \um_{\sigma,+})-T_n(
            \um_{\sigma,-})),
\end{align*}
and $|T_2| = -T_{2,1} + T_{2,2}$.\\
As far as $-T_{2,1}$ is concerned, we observe that for any $\sigma \in \A$,
$|\um_{\sigma,+}|\ge n$ implies $|\um_{\sigma,-}|\ge n$. To deal with this term, as
in \cite{leclavier}, we split the sum on $\{|\um_{\sigma,+}|\le r\}$ and on
$\{r\le|\um_{\sigma,+}|\le n\}$ where $r$ is a positive real number which will be
chosen later. We have
\[
  -T_{2,1}=I_1+I_2
\]
with
\begin{align}
I_1 &= \frac{1}{n}\sum\limits_{\underset{|\um_{\sigma,+}|\le r}{\sigma\in\A}}
      \m(\sigma)\,|\vv_{K,\sigma}|\,\um_{\sigma,+}(T_n(\um_{\sigma,-})-T_n(\um_{\sigma,+})), \label{I1}
  \\
I_2 &= \frac{1}{n}\sum\limits_{\underset{r\le|\um_{\sigma,+}|\le
      n}{\sigma\in\A}}
      \m(\sigma)\,|\vv_{K,\sigma}|\,\um_{\sigma,+}(T_n(\um_{\sigma,-})-T_n(\um_{\sigma,+})). \label{I2}
\end{align}
Recalling that $\left( \sum\limits_{\sigma\in\A}
  \m(\sigma)d_\sigma|\vv_{K,\sigma}|^2\right)^\frac{1}{2}$ is bounded by
$d^{\frac{1}{2}}\|\vv\|_{(L^2(\Omega))^d}$, the Cauchy-Schwarz inequality and
the Young inequality yield that
\begin{align}\label{|I1|}
|I_1| &= \frac{1}{n}\sum_{\substack{\sigma\in\A \\ |\um_{\sigma,+}|\le r}}
  \m(\sigma)\,|\vv_{K,\sigma}|\,\um_{\sigma,+}
  (T_n(\um_{\sigma,-})-T_n(\um_{\sigma,+})) \nonumber \\
&\le \frac{1}{n} \left(\sum_{\substack{\sigma\in\A \\ |\um_{\sigma,+}|\le
  r}}\m(\sigma) d_{\sigma} |\vv_{K,\sigma}|^2 \right)^{\frac{1}{2}}
  \left(\sum_{\substack{\sigma\in\A \\ |\um_{\sigma,+}|\le r}}
  \frac{\m(\sigma)}{d_{\sigma}}(\um_{\sigma,+})^2
  (T_n(\um_{\sigma,-})-T_n(\um_{\sigma,+}))^2\right)^{\frac{1}{2}} \nonumber \\
&\le \frac{r}{n} \left(\sum_{\substack{\sigma\in\A \\ |\um_{\sigma,+}|\le
  r}}\m(\sigma) d_{\sigma} |\vv_{K,\sigma}|^2 \right)^{\frac{1}{2}}
  \left(\sum_{\substack{\sigma\in\A \\ |\um_{\sigma,+}|\le r}}
  \frac{\m(\sigma)}{d_{\sigma}} (\um_K-\um_L)(T_n(\um_K)-T_n(\um_L))\right)^{\frac{1}{2}}
  \nonumber \\
&\le \frac{1}{n} \left[ \frac{r^2 d\, \|\,\vv \,\|_{(L^2(\Omega))^d}^2}{2\beta}
                  + \frac{\beta}{2} \sum_{\substack{\sigma\in\A \\
  |\um_{\sigma,+}|\le r}} \frac{\m(\sigma)}{d_{\sigma}}
  (\um_K-\um_L)(T_n(\um_K)-T_n(\um_L)) \right],
\end{align}
where $\beta>0$ (to be chosen later).
To control the second term $I_2$, we distinguish the case $d\geq 3$ and
$d=2$. We know that for any $\sigma\in\A$, $|\um_{\sigma,+}|\ge r$ implies
$|\um_{\sigma,-}|\ge r$; if $d\geq 3$ the equality
$$\frac{1}{d}+\frac{d-2}{2d}+\frac{1}{2}=1,$$
and the  Hölder inequality give
\begin{equation*}
\begin{split}
|I_2| & \le \frac{1}{n} \sum_{\substack{\sigma\in\A \\ r\le|\um_{\sigma,+}|\le n
    \\ r\le|\um_{\sigma,-}|}} \m(\sigma)\, |\vv_{K,\sigma}|\,T_n(\um_{\sigma,+})\,
(T_n(\um_{\sigma,-})-T_n(\um_{\sigma,+})) \\
      & \le \frac{1}{n}\Bigg(\sum_{\substack{\sigma\in\A \\
          r\le|\um_{\sigma,+}|\le n \\ r\le|\um_{\sigma,-}|}} \m(\sigma)\,
      d_{\sigma}\,|\vv_{K,\sigma}|^d\Bigg)^{\frac{1}{d}}
      \Bigg(\sum_{\substack{\sigma\in\A \\ r\le|\um_{\sigma,+}|\le n \\
          r\le|\um_{\sigma,-}|}} \m(\sigma)\,
      d_{\sigma}|T_n(\um_{\sigma,+})|^{\frac{2d}{d-2}}\Bigg)^{\frac{d-2}{2d}} \\
      & \ \ \ \ \ \times \Bigg(\sum_{\substack{\sigma\in\A \\
          r\le|\um_{\sigma,+}|\le n \\ r\le|\um_{\sigma,-}|}} \frac{\m(\sigma)}{
        d_{\sigma}}(T_n(\um_{\sigma,-})-T_n(\um_{\sigma,+}))^2\Bigg)^{\frac{1}{2}}\\
      & \le \frac{1}{n}\Bigg(\sum_{\substack{\sigma\in\A \\
          r\le|u_{\sigma,+}|\le n \\ r\le|\um_{\sigma,-}|}}   \m(\sigma)\,
      d_{\sigma}\,|\vv_{K,\sigma}|^d\Bigg)^{\frac{1}{d}}
      \Bigg(\sum_{\sigma\in\E_{int}} \m(\sigma)\,
      d_{\sigma}|T_n(\um_{\sigma,+})|^{\frac{2d}{d-2}}\Bigg)^{\frac{d-2}{2d}} \\
      & \ \ \ \ \ \times \Bigg(\sum_{\sigma\in\E_{int}}\frac{\m(\sigma)}{
        d_{\sigma}}(T_n(\um_{\sigma,-})-T_n(\um_{\sigma,+}))^2\Bigg)^{\frac{1}{2}}.\\
\end{split}
\end{equation*}
Recalling that $\underline{\med}(T_n(u_m))=0$, the discrete  Poincaré-Wirtinger inequality \eqref{PW} and the discrete Sobolev inequality \eqref{Sobolev} lead to
\begin{equation}\label{|I2|}
|I_2| \le C_1\frac{d^{\frac{1}{2}} \|\vv\|_{(L^d(E_r))^d}
}{n}\sum_{\sigma\in\E_{int}} \frac{\m(\sigma)}{d_{\sigma}}\,
(T_n(\um_K)-T_n(\um_L))^2,
\end{equation}
where  $E_r$ is the set where $|\um_{\sigma,+}|\geq r$ and $C_1>0$ is a constant
independent of $n$ and $\M$. If $d=2$, similar arguments lead to
\begin{equation}\label{|I2|bis}
|I_2| \le C_2 \frac{d^{\frac{1}{2}} \|\vv\|_{(L^p(E_r))^d}
}{n}\sum_{\sigma\in\E_{int}} \frac{\m(\sigma)}{d_{\sigma}}\,
(T_n(\um_K)-T_n(\um_L))^2,
\end{equation}
where $C_2>0$ is a constant independent of $n$ and $\M$.\\
In view of Corollary~\ref{log} and since $\vv\in (L^p(\Omega))^d$ ($2<p<+\infty$
if $d=2$, $p=d$ if $d\ge3$), the absolute continuity of the integral implies
that there exists $r>0$ (independent of $m$) such that for all $m$
\begin{equation}\label{rchoice}
d^{\frac{1}{2}}\|\vv\|_{(L^p(E_r))^d}  \le \frac{1}{2}.
\end{equation}
Then from (\ref{|I2|}), (\ref{|I2|bis}) and (\ref{rchoice}) we obtain
\begin{equation}\label{|I2|2}
|I_2|\le \frac{C_3}{2n}\sum_{\sigma\in\E_{int}} \frac{\m(\sigma)}{d_{\sigma}}\,
(\um_K-\um_L)(T_n(\um_K)-T_n(\um_L)).
\end{equation}
Recalling that $-T_2 \leq - T_{2,1}$, the inequalities \eqref{|I2|2} and \eqref{|I1|} lead to
\begin{align}\label{term}
-T_2 & \le |I_1+I_2|\nonumber\\ & \le \frac{1}{n} \frac{r^2 d\, \|\,\vv
                                  \,\|_{(L^2(\Omega))^d}^2}{2\beta}
                                  +\frac{C_4}{2n} \sum_{\sigma\in\E_{int}}
                                  \frac{\m(\sigma)}{d_{\sigma}}
                                  (\um_K-\um_L)(T_n(\um_K)-T_n(\um_L)) ,
\end{align}
where $C_4$ is a positive constant depending on $\beta$ and $C_3$. Since $0 <
\mu \leq \lambda(u_{m})$, we choose $\beta >0$ such that the second term of the
right-hand side of \eqref{term} is $\leq \dfrac{T_1}{2}$. It follows that
\begin{align}\label{-T2 2}
-T_2 & \leq \frac{1}{n} \frac{r^2 d\, \|\,\vv \,\|_{(L^2(\Omega))^d}^2}{2\beta} + \dfrac{T_1}{2}
\nonumber \\
& \leq R(n,h_{\T_{m}}) + \dfrac{T_1}{2},
\end{align}
%&\le \frac{1}{n} \frac{r^2 d\, \|\,\vv \,\|_{L^2(\Omega)^d}^2}{2\beta}  + \frac{C}{2} T_1 \nonumber \\
%&\le R(n, h_{\T}) + \frac{C}{2} T_1,
with $R$ verifying
$\displaystyle\lim_{n\rightarrow+\infty}\varlimsup_{h_{\T_{m}}\rightarrow0}R(n,h_{\T_{m}})=0$.\\
Since $T_1+T_2=T_3$, \eqref{-T2 2} allows one to conclude that
\begin{equation}
\lim_{n\rightarrow+\infty}\varlimsup_{h_{\M_m}\rightarrow0}
\frac{1}{n}\sum_{\sigma\in\E_{int}}\frac{\m(\sigma)}{d_{\sigma}}\lambda(u_{m})_{\sigma}(\um_K-\um_L)(T_n(\um_K)-T_n(\um_L))=0,
\end{equation}
which gives \eqref{energiediscrete}.
We are now in a position to prove (\ref{corenergiediscrete}). Recalling that $T_1+T_2=T_3$, we have $ T_1 + T_{2,2} \le |T_{2,1}| + T_3$. Since $T_{2,2}$ is non negative, using \eqref{limT3}, \eqref{-T2 2} we get
\begin{equation*}
T_1  \le \frac{C_5}{2}T_1 + R +T_3\ \textnormal{with}
\end{equation*}
\begin{align*}
& \lim_{n\rightarrow+\infty}\varlimsup_{h_{\T_{m}}\rightarrow 0} R=0, \\
& \lim_{n\rightarrow+\infty}\varlimsup_{h_{\T_{m}}\rightarrow 0} T_3=0.
\end{align*}
As a consequence we obtain \begin{equation}\label{fi}
  \lim_{n\rightarrow+\infty}\varlimsup_{h_{\T_{m}}\rightarrow 0} |T_{2,1}|=0.
\end{equation}
Moreover, writing again $ T_1 + T_{2,2} \le |T_{2,1}| + T_3$, \eqref{fi}, \eqref{limT3} and \eqref{energiediscrete} imply that
\begin{equation}\label{fii}
    \displaystyle\lim_{n\rightarrow+\infty}\varlimsup_{h_{\T_{m}}\rightarrow0}T_{2,2}=0.
    \end{equation}
Therefore from \eqref{fi} and \eqref{fii} we deduce \eqref{corenergiediscrete}.
\end{proof}{}

\section{Existence of a solution to the scheme}

In this section we prove  that there exists at least one solution to the
discrete scheme. Since the scheme is nonlinear we use a fixed point argument
together with the study of the linear version of our problem for which we adapt
the arguments developed in \cite{CHD11} for a linear problem with Neumann
boundary conditions and mean value.

\begin{proof}[Proof of Theorem \ref{exist}]
  The proof is divided into 2 steps. In Step 1 with the help of \cite{CHD11} we
  construct a map in view of the fixed point argument. In Step 2 using estimates
  in Proposition \ref{estlog} we conclude with the Brouwer fixed point theorem the
  existence of a solution.

  \noindent{\bfseries Step 1.}
  Let $\widetilde{u}=(\widetilde{u}_{K})_{K\in\T} \in X(\T)$ and let us consider the linear scheme
\begin{equation}\label{schemelin}
\forall K\in\M,\ \sum_{\sigma\in\E_{int}}\frac{\mathrm{m} (\sigma)}{d_{\sigma}}\lambda(\widetilde{u})_\sigma(u_K-u_L) + \sum_{\sigma\in\E_{int}}\m(\sigma) \vv_{K,\sigma} u_{\sigma,+}  = \int_K f\,\mathrm{d}x,
\end{equation}
where $u=(u_{K})_{K\in\T} \in X(\T)$ is the unknown. Following \cite{CHD11}, it can be rewritten as the linear system

\begin{equation}
  \label{systlin}
  \mA U=F
\end{equation}
where $U=(u_{K})_{K\in\mathcal{T}}$, $F=(\int_{K} f dx)_{K\in\mathcal{T}}$ and
$\mA$ is the square matrix of size $\card(\mathcal{T})\times \card(\mathcal{T})$
with
\[
  \left\{\begin{aligned}
&\mA_{K,K}=\sum_{\sigma\in\mathcal{E}_{K,int}}
\mathrm{m}(\sigma)\big(\frac{\lambda(\widetilde{u}_{\sigma})}{d_{\sigma}}  +
\vv_{K,\sigma}^{+}\big),\quad  \forall K\in\mathcal{T},\\
&\mA_{K,L}= \mathrm{m}(\sigma)\big(- \frac{\lambda(\widetilde{u}_{\sigma})}{d_{\sigma}}  -
 \vv_{K,\sigma}^{-}\big),\quad  \forall
K\in\mathcal{T},\forall L\in N(K),\text{ with }\sigma=K|L,\\
& \mA_{K,L}=0,\quad  \forall
K\in\mathcal{T},\forall L\notin N(K).
\end{aligned}
\right.
\]
At this step having $f$ belonging to  $L^{1}(\Omega)$ or $f\in L^{2}(\Omega)$
does not play any role. From Proposition~3.1 in \cite{CHD11} (see also
Remark~2.4 in \cite{CHD11} when $-\Delta u$ is replaced by $-\diw(a(x)\nabla
u)$, which
is the isotropic case)  it follows that
\begin{itemize}
\item $\dim(\ker(\mA))=1$ and any (non zero) element $U$ belonging to
  $\ker(\mA)$ verifies either $u_{K}>0$ for all $K\in\mathcal{T}$ or $u_{K}<0$
  for all $K\in\mathcal{T}$.
\item $\ker(\mA^{\top})=\R(1,1,\cdots,1)^{\top}$ and thus
  \[
    \operatorname{Im}(\mA)=\big\{ (F_{K})_{K\in\mathcal{T}}\,;\,
    \sum_{K\in\mathcal{T}} F_{K}=0\big\}.
  \]
\end{itemize}
Since $\int_{\Omega} f \dx=0$, $F$ belongs to $\operatorname{Im}(\mA)$ and then
there exists at least $\overline{U}$ solution of $\mA\overline{U}=F$. If
$V=(v_{K})_{K\in\mathcal{T}}$ denotes an element of $\ker(\mA)$ such that
$v_{k}>0$ for any $K\in\mathcal{T}$, for any $\lambda\in\R$, the vector
$\overline{U}+\lambda V$ is a solution of \eqref{systlin}. Since $v_{K}>0$,
$\forall K\in\mathcal{T}$, the function $\lambda\mapsto
\medu(\overline{U}+\lambda V)$ is continuous and increasing while
$\lim_{\lambda\rightarrow +\infty} \medu(\overline{U}+\lambda V)=+\infty$ and
$\lim_{\lambda\rightarrow -\infty} \medu(\overline{U}+\lambda V)=-\infty$. It
follows that there exists at least one solution $u\in X(\mathcal{T})$ verifying
the scheme \eqref{schemelin} and $\medu(u)=0$.

The uniqueness is a consequence of the characterization of $\ker(\mA)$.
As a conclusion we can define the map $\Gamma$  from $X(\mathcal{T})$
into $X(\mathcal{T})$ by
\[
\forall \widetilde{u}\in X(\mathcal{T}), \ \Gamma\big(\widetilde{u}\big)=u
\]
where $u\in X(\mathcal{T})$ is the unique solution of \eqref{schemelin} such that $\medu(u)=0$.
\par\smallskip
\noindent{\bfseries Step 2.} In this step we prove that
\begin{gather}
  \label{eq:exis0}
  \exists C>0,\forall \widetilde{u}\in X(\mathcal{T}),\quad
  \|\Gamma(\widetilde{u})\|_{L^{\infty}}\leq C,
  \\
  \label{eq:exis1}
  \Gamma \text{ is a continuous map}
\end{gather}
in order to apply the Brouwer fixed point theorem.

The boundedness of $\Gamma$ relies on Proposition
\ref{estlog}. Indeed by replacing $\lambda(u)_{\sigma}$ by
$\lambda(\widetilde{u})_{\sigma}$ in the proof of Proposition
\ref{estlog} it can be shown that there exists $C>0$ (not depending on
$\widetilde{u}$) such that
\begin{equation} \label{existestlogeq}
\| \ln (1 + |\Gamma(\widetilde{u})|)\|_{1,2,\M}^2 \le C\left(
  2\|f\|_{L^1(\Omega)} + d |\Omega|^{\frac{p-2}{p}}\,
  \|\,\vv\,\|^2_{\left(L^p(\Omega)\right)^d}\right),
\end{equation}
%where $|\vv|$ denotes the Euclidean norm of $\vv$ in $\R^d$. 
Since $\Gamma(\widetilde{u})$ lies in a finite dimension vector space we obtain that
\eqref{eq:exis0} holds true.

We now prove that $\Gamma$ is a continuous map. Let $(\widetilde{u}_{n})_{n\in\N}$ and
$\widetilde{u}$ belonging to $X(\mathcal{T})$ such that $\widetilde{u}_{n}$ goes
to $\widetilde{u}$ as $n$ goes to infinity. In view of \eqref{eq:exis0} up to a
subsequence, still indexed by $n$, there exists $w\in X(\mathcal{T})$ such that
$\Gamma(\widetilde{u}_{n})$ tends to $w$ as $n$ goes to infinity. Recalling that
the coefficient of the matrix $\mA$ are continuous with respect to
$\widetilde{u}$ we obtain that $w$ is a solution to the scheme
\begin{equation}\label{eq:exis2}
\forall K\in\M,\ \sum_{\sigma\in\E_{int}}\frac{\mathrm{m} (\sigma)}{d_{\sigma}}\lambda(\widetilde{u})_\sigma(w_K-w_L) + \sum_{\sigma\in\E_{int}}\m(\sigma) \vv_{K,\sigma} w_{\sigma,+}  = \int_K f\,\mathrm{d}x.
\end{equation}
The mesh $\mathcal{T}$ being fixed since we consider constant piecewise
functions the fact that $\medu\big(\Gamma(\widetilde{u}_{n})\big)=0$ for any $n\in \N$ implies
that  $\medu(w)=0$. Recalling that  $u=\Gamma(\widetilde{u})$ is the unique
solution of \eqref{schemelin} with $\medu(u)=0$ we conclude that $u=w$. Since
the limit does not depend on the subsequence we obtain that $\Gamma$ is a
continuous map.

In view of \eqref{eq:exis0} and \eqref{eq:exis1} the Brouwer fixed point theorem allows one to conclude the proof of Theorem \ref{exist}.
\end{proof}

\section{Convergence of the scheme}

In this section we prove using Section 4 that the function $u$ obtained in
Proposition \ref{EstTn} is a renormalized solution to equation \eqref{Pb1}. The
main difficulty is the equation \eqref{def3}: in short the test functions in
\eqref{def3} are of the kind $\varphi S(u)$ which is nonlinear in $u$. The
strategy is to mix renormalized techniques and finite volume with the use of a
discrete version of $\varphi S_n(u_{\mathcal{M}_m})$ (see below the definition
of $S_n$) where $\varphi$ belongs to $\mathcal{D}(\overline{\Omega})$ (see
\cite{leclavier} in the case of Dirichlet boundary condition). Before proving
Theorem \ref{conv} we give in Lemma \ref{lemme6.1} a convergence result
concerning $S_n(u_{\mathcal{M}_m})$ (see e.g. \cite{LL12,leclavier})  and in
Corollary \ref{corol1} the asymptotic behavior of an extra term (with respect to
the continuous case) which appears when we pass to the limit in the discrete
equation.

%The aim of this section is to prove the convergence of the solution of the finite volume scheme given by Theorem \ref{exist} toward the renormalized solution $u$ with $\med(u)=0$ to the continuous problem. More precisely, by using the results of Section 4 and passing to the limit in the scheme, we show that $u$ satisfies conditions \eqref{def0}-\eqref{def3} of definition \ref{def}.

First, let us define the function $S_n$, $n\geq 1$,
%which is used for the discrete version of the test function in the renormalized formulation, and then prove a convergence result concerning the function $h_n$ and a corollary which is necessary to pass to the limit in the diffusion term.
by
\begin{equation}\label{S_n}
S_n(s)=\left\{ \begin{aligned} & 0 &\text{ if } s\le-2n, \\
& \frac{s}{n}+2 & \text{ if } -2n\le s\le -n,  \\
& 1 &\text{ if } -n\le s\le n, \\
& \frac{-s}{n}+2 & \text{ if } n\le s\le 2n,  \\
& 0 &\text{ if } s\ge 2n. \end{aligned}\right.
\end{equation}

\begin{figure}[!h]
	\begin{center}
		\begin{tikzpicture}
		\draw[->] (-6,0) -- (6,0);
		\draw (6,0) node[right] {$s$};
		\draw [->] (0,-1) -- (0,2);
		\draw (0,2) node[right] {$S_n(s)$};
		\draw[domain=-6:-4] plot(\x,0);
		\draw[domain=-4:-2] plot(\x,{(4-abs(\x))/2});
		\draw[domain=-2:2] plot(\x,1);
		\draw[domain=2:4] plot(\x,{(4-abs(\x))/2});
		\draw[domain=4:6] plot(\x,0);
		\draw (-4,0) node[below] {$-2n$};
		\draw (-2,0) node[below] {$-n$};
		\draw (2,0) node[below] {$n$};
		\draw (4,0) node[below] {$2n$};
		\draw (0,1.2) node[above,right] {$1$};
		\end{tikzpicture}
		\caption{The function $S_n$} \label{fig:h_n}
	\end{center}
\end{figure}
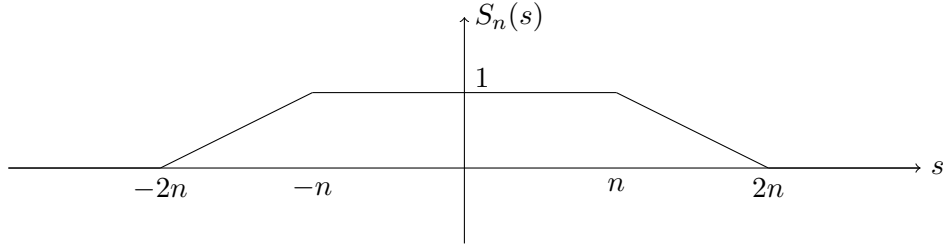

\begin{lemma}\label{lemme6.1}
Let $(\M_m)_{m\ge1}$ be a sequence of admissible meshes satisfying \eqref{cmesh}
such that $h_{\M_{m}}\rightarrow 0$ as $m\rightarrow \infty$
and let $u_{m}=(\um_{K})_{K\in\T_{m}}\in X(\T_m)$ be a sequence of solution of (\ref{scheme}) such
that the conclusions of Proposition \ref{EstTn} hold true.
We define the function $\Snbm$ over the diamonds by
\begin{align*}
\forall\sigma=K|L\in\E_{int},\forall x\in D_{\sigma}, & \quad \overline{S}^m_n(x)=\frac{S_n(\um_K)+S_n(\um_L)}{2},\\
\forall\sigma\in\E_{ext},\forall x\in D_{\sigma},& \quad \overline{S}^m_n(x)=S_n(\um_K),
% \tilde{\lambda}(x)=\lambda(u_{\M_m})_\sigma
\end{align*}
and the function $\overline{\lambda}^m$ by
\begin{equation*}
\forall\sigma\in\E,\forall x\in D_{\sigma}    \quad\overline{\lambda}^{m}(x)=\lambda(u_{m})_\sigma.
\end{equation*}
 Then the functions $\Snbm$ and $\overline{\lambda}$ converge respectively to $S_n(u)$ and $\lambda(u)$ in $L^q(\Omega)\ \forall q\in[1,+\infty[$ and in $L^\infty(\Omega)$ weak-$*$, as $h_{\M_m}\rightarrow0$, where $u$ is the limit of $u_{\M_m}$.
\end{lemma}

\begin{proof}
Since the functions $\Snbm$ is bounded with respect to $m$ it is sufficient to
prove that the convergence holds true in $L^2(\Omega)$. By Proposition
\ref{EstTn} $u_{m}$ converges to $u$ a.e. in $\Omega$. Since $S_n$ is a
continuous and bounded function, the Lebesgue dominated convergence theorem
gives that $S_n(u_{m})\rightarrow S_n(u)$  in $L^2(\Omega)$, so that it is
sufficient to study the behavior of $ \Snbm - S_n(u_{m})$. Recalling that
$S_n$ is Lipschitz continuous and has a compact support, we have
\begin{align*}
\| \Snbm - S_n(u_{m}) \|^2_{L^2(\Omega)} &= \int_{\Omega} | \Snbm - S_n(u_{m}(x)) |^2\, \mathrm{d}\,x \\
&= \sum_{\sigma\in\E}\int_{D_{\sigma}}| \Snbm(x) - S_n(u_{m}(x)) |^2\, \mathrm{d}\,x \\
&= \sum_{\sigma\in\E_{int}}\int_{D_{\sigma}}\left|
    \frac{S_n(\um_K)+S_n(\um_L)}{2} - S_n(u_{m}(x)) \right|^2\, \mathrm{d}\,x \\
%&= \frac{1}{2^q} \sum_{\sigma\in\E} \m(D_{\sigma})\, |h_n(u_K)-h_n(u_L)|^q \\
&= \frac{1}{4}\sum_{\sigma\in\E_{int}} \m(D_{\sigma})\, |S_n(\um_K)-S_n(\um_L)|^2 \\
&\le \frac{1}{4n^2}\sum_{\sigma\in\E_{int}} \m(D_{\sigma}) |T_{2n}(\um_K)-T_{2n}(\um_L)|^2 \\
%&= \frac{1}{4}\sum_{\sigma\in\E} \frac{\m(\sigma)d_{\sigma}}{d} |T_{2n}(u_K)-T_{2n}(u_L)|^2 \\
&= \frac{1}{4dn^2} \sum_{\sigma\in\E_{int}}\m(\sigma)d_{\sigma}\, \left|\frac{T_{2n}(\um_K)-T_{2n}(\um_L)}{d_{\sigma}}\right|^2 \, (d_{\sigma})^2 \\
&\le \frac{1}{4dn^2} |T_{2n}(u_{m})|_{1,2,\M}^2\, (h_{\M_m})^2,
\end{align*}
so that $\| \Snbm - S_n(u_{m}) \|^2_{L^2(\Omega)}$ goes to zero as $h_{\M_m}\rightarrow0$.

As far as $\labm$ is concerned let us first define the function $\overline{u}^m$ by
\begin{align*}
    \forall\sigma=K|L\in\E_{int},  \forall x\in D_{\sigma}, &\quad \overline{u}^m(x)=\frac{\um_K+\um_L}{2},\\
    \forall\sigma\in\E_{ext},  \forall x\in D_{\sigma},
    &\quad \overline{u}^m(x)=\um_K.
\end{align*}
In view of the definition of $\lambda(u)_\sigma$ we have in $\Omega$
 \begin{equation}
\min(\lambda(u_{m}),\lambda(2\overline{u}^m-u_{m}))\leq \overline{\lambda}^m \leq \max(\lambda(u_{m}),\lambda(2\overline{u}^m-u_{m})).
\end{equation}
In view of already used arguments for $\overline{S}^m$, since the function $T_n$
is Lipschitz continuous, $T_n(\overline{u}^m)$ converges to $T_n(u)$ in
$L^q(\Omega)$, $\forall q\in[1,+\infty[$ and in $L^\infty$ weak-$*$. By the
diagonal process, up to a subsequence still index by $m$, $\overline{u}^m$ goes
to $u$ a.e. in $\Omega$ as $h_{\M_m}$ goes to zero. Since $\lambda$ is a bounded
continuous function we obtain, up to a subsequence, that $\overline{\lambda}^m$
converges to  $\lambda(u)$ in $L^q(\Omega)$, $\forall q\in[1,+\infty[$ and in
$L^\infty$ weak-$*$ as $h_{\M_m}$ goes to zero. To conclude it is sufficient to
observe that the limit of $\overline{\lambda}^m$ is independent of the
subsequence.
\end{proof}

\begin{corollary}\label{corol1}
Let $(\M_m)_{m\ge1}$ be a sequence of admissible meshes satisfying \eqref{cmesh}
such that $h_{\M_{m}}\rightarrow 0$ as $m\rightarrow \infty$.
Let $u_{m}=(\um_{K})_{K\in\T_{m}}\in X(\T_m)$ be a sequence of solution of (\ref{scheme}) such
that the conclusions of Proposition \ref{EstTn} hold true. Then we have
\begin{equation}\label{corol2}
\lim_{n\rightarrow+\infty}\varlimsup_{h_{\M_{m}}\rightarrow 0}
\sum_{\substack{\sigma\in\E_{int} \\ |\um_K|\le 2n \\ |\um_L|>4n}}
\lambda(u)_\sigma\frac{\m(\sigma)}{d_{\sigma}} |\um_L| =0.
\end{equation}

\end{corollary}
\begin{proof}
 For $m\in \N$, let us consider $K\in\T_{m}$ and $n\in\N$.
%Due to \eqref{energiediscrete}, we have
%\begin{equation*}
%\lim_{n\rightarrow+\infty}\lim_{h_{\M}\rightarrow0}\frac{1}{4n}\sum_{\sigma\in\E}\lambda(u)_\sigma \frac{|\sigma|}{d_{\sigma}} (u_L-u_K) (T_{4n}(u_K)-T_{4n}(u_L))=0.
%\end{equation*}
On one hand if $|\um_K|\le 2n$ and   $\um_L>4n$ then
\[
  (\um_K-\um_L)(T_{4n}(\um_K)-T_{4n}(\um_L)) \ge \frac{\um_L}{2} 2n \ge 0.
\]
On the other hand, if $|\um_K|\le 2n$ and $\um_L<-4n$ then
\[
  (\um_K-\um_L)(T_{4n}(\um_K)-T_{4n}(\um_L)) \ge \frac{-\um_L}{2} 2n \ge 0.
\]
It follows that
\begin{equation}\label{ineq1}
\sum_{\substack{\sigma\in\E_{int} \\ |\um_K|\le 2n \\ |\um_L|>4n}}
\frac{\m(\sigma)}{d_{\sigma}}\lambda(u_{m})_\sigma |\um_L| \le
\frac{1}{4n}\sum_{\sigma\in\E_{int}}
\frac{\m(\sigma)}{d_{\sigma}}\lambda(u_{m})_\sigma (\um_L-\um_K)
(T_{4n}(\um_K)-T_{4n}(\um_L)).
\end{equation}
Using the discrete estimate on energy  \eqref{energiediscrete}, we have
\begin{equation} \label{ineq2}
\lim_{n\rightarrow+\infty}\lim_{h_{\M_{m}}\rightarrow0}\frac{1}{4n}\sum_{\sigma\in\E_{int}}\lambda(u_{m})_\sigma
\frac{\m(\sigma)}{d_{\sigma}} (\um_L-\um_K) (T_{4n}(\um_K)-T_{4n}(\um_L))=0,
\end{equation}
so that \eqref{ineq1} and \eqref{ineq2} give \eqref{corol2}.
\end{proof}

We are now in a position to prove Theorem \ref{conv}.

\begin{proof2}{\bf{of Theorem \ref{conv}}}
Let $m\geq 1$ and let us consider $u_m=(\um_K)_{K\in \T_m}$ be  a solution of the scheme \eqref{scheme}.

For $\varphi$ a function belonging to  $\mathcal{D}(\overline{\Omega})$ we  denote by $\varphi_{m}$ the function defined by $\varphi_K=\varphi(x_K)$ for all $K\in\T_m$.
\noindent For $n\in\N$, multiplying each equation of the scheme by $\varphi(x_K)S_n(\um_K)$ (which is a discrete version of the test function used in the renormalized formulation), summing over the control volumes and gathering by edges, we get $$T_1+T_2=T_3$$ with
\begin{align*}
T_1 &= \sum_{\sigma\in\E_{int}}\frac{\m(\sigma)}{d_{\sigma}}\lambda(u_m)_\sigma(\um_K-\um_L)(\varphi(x_K)S_n(\um_K)-\varphi(x_L)S_n(\um_L)), \\
T_2 &= \sum_{\sigma\in\E_{int}} \m(\sigma) \vv_{K,\sigma} \um_{\sigma,+}(\varphi(x_K)S_n(\um_K)-\varphi(x_L)S_n(\um_L)), \\
T_3 &= \sum_{K\in\T} \int_K f\, \varphi(x_K)S_n(\um_K).
\end{align*}

  Since $S_n(u_m)\rightarrow S_n(u)$ a.e. and $L^{\infty}$ weak $\star$, by the
  regularity of $\varphi$, $\varphi_{m}\rightarrow\varphi$ uniformly and
  $|f\,\varphi_{m}\,S_n(u_{m})|\le C_{\varphi} |f|\in L^1(\Omega)$, the
  Lebesgue theorem ensures that
\begin{equation}\label{cvT3}
T_3 = \int_{\Omega}f\,\varphi_{m}\,S_n(u_{m})\,\mathrm{d}\,x \xrightarrow[h_{{\M}_m}\rightarrow0]\ \int_{\Omega}f\,\varphi\,S_n(u)\,\mathrm{d}\,x.
\end{equation}

We now study the convergence of the diffusion term. We write
\begin{align*}
T_1 &= \sum_{\sigma\in\E_{int}} \frac{\m(\sigma)}{d_{\sigma}}\lambda(u_m)_\sigma (\um_K-\um_L) (\varphi(x_K)S_n(\um_K)-\varphi(x_L)S_n(\um_L)) \\
&= T_{1,1} + T_{1,2}
\end{align*}
with
\begin{align*}
T_{1,1} &= \sum_{\sigma\in\E_{int}} \frac{\m(\sigma)}{d_{\sigma}}\lambda(u_m)_\sigma  S_n(\um_K)\,(\um_K-\um_L)\,(\varphi(x_K)-\varphi(x_L)), \\
T_{1,2} &= \sum_{\sigma\in\E_{int}} \frac{\m(\sigma)}{d_{\sigma}}\lambda(u_m)_\sigma  \varphi(x_L)\, (\um_K-\um_L)\, (S_n(\um_K)-S_n(\um_L)).
\end{align*}
According to the definition of $S_n$ we have $$|T_{1,2}|\le\frac{1}{n}\sum_{\sigma\in\E_{int}} \frac{\m(\sigma)}{d_{\sigma}}\lambda(u_m)_\sigma \varphi(x_L)\, (\um_K-\um_L)\, (T_{2n}(\um_K)-T_{2n}(\um_L)),$$
so that \eqref{energiediscrete} give
\begin{equation}\label{cvT1,2}
\lim_{n\rightarrow+\infty}\varlimsup_{h_{{\M}_m}\rightarrow 0} T_{1,2}=0.
\end{equation}

The main difference with respect to the continuous case is that $\um_K$ is truncated while $u_L$ is not in $T_{1,1}$. To control this term we have to write
\begin{align*}
T_{1,1} &=  \sum_{\sigma\in\E_{int}} \frac{\m(\sigma)}{d_{\sigma}}\lambda(u_m)_\sigma  S_n(\um_K)\,(T_{2n}(\um_K)-\um_L)\,(\varphi(x_K)-\varphi(x_L)), \\
&= I + II + III
\end{align*}
with
\begin{align*}
I &= \sum_{\sigma\in\E_{int}} \frac{\m(\sigma)}{d_{\sigma}}\lambda(u_m)_\sigma \frac{S_n(\um_K)+S_n(\um_L)}{2}\,(T_{4n}(\um_K)-T_{4n}(\um_L))\,(\varphi(x_K)-\varphi(x_L)), \\
II &= \sum_{\sigma\in\E_{int}} \frac{\m(\sigma)}{d_{\sigma}}\lambda(u_m)_\sigma S_n(\um_K)\,(T_{4n}(\um_L)-\um_L)\,(\varphi(x_K)-\varphi(x_L)), \\
III &=\sum_{\sigma\in\E_{int}}\frac{\m(\sigma)}{d_{\sigma}}\lambda(u_m)_\sigma\frac{S_n(\um_K)-S_n(\um_L)}{2}\,(T_{4n}(\um_K)-T_{4n}(\um_L))\,(\varphi(x_K)-\varphi(x_L))
\end{align*}
We first study the asymptotic behavior of $II$ and $III$ as the parameter $h_{{\M}_n}$ goes to zero. Since
\begin{align*}
|II| &\le \sum_{\sigma\in\E_{int}} \frac{\m(\sigma)}{d_{\sigma}}\lambda(u)_\sigma |S_n(\um_K)|\,|T_{4n}(\um_L)-\um_L|\,|\varphi(x_K)-\varphi(x_L)| \\
&\le 2 \|\varphi\|_{L^{\infty}(\Omega)} \sum_{\substack{\sigma\in\E_{int} \\ |\um_K|\le 2n \\ |\um_L|>4n}} \frac{\m(\sigma)}{d_{\sigma}} \lambda(u_m)_\sigma|\um_L|,
\end{align*}
and due to Corollary \ref{corol1} we obtain that
\begin{equation}\label{cvII}
\varlimsup_{h_{\M}\rightarrow 0} |II|\leq \omega(n) \|\varphi\|_{L^\infty(\Omega)},
\end{equation}
where $\omega(n)$ tends to zero as $n$ goes to infinity. In the sequel of the present proof $\omega(n)$ is a generic positive quantity such that $\lim_{n \rightarrow + \infty} w(n)=0$.
According to the definition of $S_n$ we have
\begin{equation*}
| III | \le \frac{\|\varphi\|_{L^{\infty}(\Omega)}}{n} \sum_{\sigma\in\E_{int}} \frac{\m(\sigma)}{d_{\sigma}}\lambda(u_m)_\sigma |\um_K-\um_L|\, |T_n(\um_K)-T_n(\um_L)|,
\end{equation*}
recalling \eqref{energiediscrete} of Proposition \ref{propestdiscrete} we obtain that
\begin{equation}\label{cvIII}
\varlimsup_{h_{\T}\rightarrow 0} |III| \leq \omega(n) \|\varphi\|_{L^\infty(\Omega)}.
\end{equation}
We now turn to $I$. By rewriting $I$ as integral over the diamonds $D_\sigma$ (see e.g. \cite{zbMATH01970883,zbMATH02027732,zbMATH05490416})
we have
\begin{align*}
I = {}& \sum_{\sigma\in\E_{int}} \m(\sigma) d_{\sigma}\lambda(u_m)_\sigma \frac{S_n(\um_K)+S_n(\um_L)}{2}\, \frac{T_{4n}(\um_K)-T_{4n}(\um_L)}{d_{\sigma}}\, \frac{\varphi(x_K)-\varphi(x_L)}{d_{\sigma}} \\
= {}& \sum_{\sigma\in\E_{int}} \int_{D_{\sigma}} \labm(x)\Snbm(x) \nabla_{{\M}_m} T_{4n}(u_{m}) \cdot \nabla \varphi\, \mathrm{d}\,x \\
&{}  \begin{multlined}
{} + \sum_{\sigma\in\E_{int}} d \m(D_{\sigma})\lambda(u_m)_\sigma \frac{S_n(\um_K)+S_n(\um_L)}{2}\, \frac{T_{4n}(\um_K)-T_{4n}(\um_L)}{d_{\sigma}}
\\ \times
\left[\frac{\varphi(x_K)-\varphi(x_L)}{d_{\sigma}} + \frac{1}{\m(D_{\sigma})} \int_{D_{\sigma}} \nabla\varphi \cdot \eta_{K,\sigma}\, \mathrm{d}\,x\right]
\end{multlined}
\\
= {} & I_{1} + I_{2}.
\end{align*}
By Lemma \ref{lemme6.1} $\Snbm \rightarrow S_n(u)$ in $
L^q(\Omega)$ for all $q \in [1,+\infty[$ and $\labm \rightarrow\lambda(u)$ in
$L^\infty $ weak $ \star$, while $\nabla_{{\M}_{m}}T_{4n}(u_m)$ tends to $\nabla
T_{4n}(u)$ weakly in $(L^2(\Omega))^d$. Since $\varphi$ belongs to
$C^\infty_c(\Omega)$ we conclude that
\begin{equation*}\label{cvI1}
    \lim_{h_{{\M}_m}\rightarrow 0} I_1 = \int_\Omega \lambda(u)S_n(u)\nabla T_{4n}(u) \cdot \nabla \varphi \,\mathrm{d} x.
\end{equation*}
By the regularity of $\varphi$ we see that
\begin{align*}
|I_2| &\le \sum_{\sigma\in\E_{int}} d \m(D_{\sigma})\lambda(u_m)_\sigma \left|\frac{S_n(\um_K)+S_n(\um_L)}{2}\right| \frac{|T_{4n}(\um_K)-T_{4n}(\um_L)|}{d_{\sigma}}
\\
& \qquad\qquad\qquad \times\left|\frac{\varphi(x_K)-\varphi(x_L)}{d_{\sigma}} + \frac{1}{\m(D_{\sigma})} \int_{D_{\sigma}} \nabla\varphi \cdot \eta_{K,\sigma}\, d\,x\right| \\
&\le \lambda_\infty \|\varphi\|_{W^{2,\infty}(\Omega)} h_{\M} \|T_{4n}(u_{m})\|_{1,1,\M},
\end{align*}
thus
\begin{equation}\label{cvI2}
\lim_{h_{{\M}_m}\rightarrow 0} I_2 =0.
\end{equation}
%Since
%\begin{align*}
%|II| &\le \sum_{\sigma\in\E} \frac{\m(\sigma)}{d_{\sigma}}\lambda(u)_\sigma |h_n(u_K)|\,|T_{4n}(u_L)-u_L|\,|\varphi(x_K)-\varphi(x_L)| \\
%&\le 2 \|\varphi\|_{L^{\infty}(\Omega)} \sum_{\substack{\sigma\in\E \\ |u_K|\le 2n \\ %|u_L|>4n}} \frac{\m(\sigma)}{d_{\sigma}} \lambda(u)_\sigma|u_L|,
%\end{align*}
%and due to Corollary \ref{corol1} we obtain that
%\begin{equation}\label{cvII}
%\lim_{n\rightarrow+\infty}\varlimsup_{h_{\T}\rightarrow 0} II=0.
%\end{equation}
%According to the definition of $h_n$ we have
%\begin{equation*}
%| III | \le \frac{\|\varphi\|_{L^{\infty}(\Omega)}}{n} \sum_{\sigma\in\E} \frac{\m(\sigma)}{d_{\sigma}}\lambda(u)_\sigma |u_K-u_L|\, |T_n(u_K)-T_n(u_L)|,
%\end{equation*}
%recalling \eqref{energiediscrete} we obtain that
%\begin{equation}\label{cvIII}
%\lim_{n\rightarrow+\infty}\varlimsup_{h_{\T}\rightarrow 0} III=0.
%\end{equation}

We now study the convergence of the convection term $T_2$. We have
\begin{align*}
T_2 &= \sum_{\sigma\in\E_{int}} \m(\sigma) \vv_{K,\sigma} \um_{\sigma,+}(\varphi(x_K)S_n(\um_K)-\varphi(x_L)S_n(\um_L)) \\
&= \sum_{\substack{\sigma\in\E_{int} \\ \vv_{K,\sigma}\ge 0}} \m(\sigma)\vv_{K,\sigma} \um_{\sigma,+}(\varphi(x_K)S_n(\um_{\sigma,+})-\varphi(x_L)S_n(\um_{\sigma,-})) \\
&\ \ \ \ + \sum_{\substack{\sigma\in\E_{int} \\ \vv_{K,\sigma}< 0}} \m(\sigma)\vv_{K,\sigma} \um_{\sigma,+}(\varphi(x_K)S_n(\um_{\sigma,-})-\varphi(x_L)S_n(\um_{\sigma,+})) \\
&= \sum_{\substack{\sigma\in\E_{int} \\ \vv_{K,\sigma}\ge 0}} \m(\sigma)\vv_{K,\sigma} \um_{\sigma,+}S_n(\um_{\sigma,+})(\varphi(x_K)-\varphi(x_L)) \\
&\ \ \ \ + \sum_{\substack{\sigma\in\E_{int} \\ \vv_{K,\sigma}\ge 0}} \m(\sigma)\vv_{K,\sigma} \um_{\sigma,+}\varphi(x_L)(S_n(\um_{\sigma,+})-S_n(\um_{\sigma,-})) \\
&\ \ \ \ - \sum_{\substack{\sigma\in\E_{int} \\ \vv_{K,\sigma}< 0}} \m(\sigma)\vv_{K,\sigma} \um_{\sigma,+}S_n(\um_{\sigma,+})(\varphi(x_L)-\varphi(x_K)) \\
&\ \ \ \ - \sum_{\substack{\sigma\in\E_{int} \\ \vv_{K,\sigma}< 0}} \m(\sigma)\vv_{K,\sigma} \um_{\sigma,+}\varphi(x_K)(S_n(\um_{\sigma,+})-S_n(\um_{\sigma,-})) \\
&= T_{2,1}+T_{2,2}+T_{2,3}
\end{align*}
with
\begin{align*}
T_{2,1} &=\sum_{\sigma\in\E_{int}} \m(\sigma)\vv_{K,\sigma} \um_{\sigma,+}S_n(\um_{\sigma,+})(\varphi(x_K)-\varphi(x_L)) \\
T_{2,2} &=\sum_{\substack{\sigma\in\E_{int} \\ \vv_{K,\sigma}\ge 0}} \m(\sigma)\vv_{K,\sigma} \um_{\sigma,+}\varphi(x_L)(S_n(\um_{\sigma,+})-S_n(\um_{\sigma,-})) \\
T_{2,3} &= - \sum_{\substack{\sigma\in\E_{int} \\ \vv_{K,\sigma}< 0}} \m(\sigma)\vv_{K,\sigma} \um_{\sigma,+}\varphi(x_K)(S_n(\um_{\sigma,+})-S_n(\um_{\sigma,-}))
\end{align*}
Since
\[
  |T_{2,2}+T_{2,3}| \le
  \frac{\|\varphi\|_{L^{\infty}(\Omega)}}{n}\sum_{\sigma\in\E_{int}}
  \m(\sigma)\,|\vv_{K,\sigma}|\,
  |\um_{\sigma,+}|\,|T_{2n}(\um_{\sigma,+})-T_{2n}(\um_{\sigma,-})|,
\]
we deduce from \eqref{corenergiediscrete} that
\begin{equation}\label{cvT2,2+T2,3}
\lim_{n\rightarrow+\infty}\varlimsup_{h_{{\M}_{m}}\rightarrow 0}T_{2,2}+T_{2,3}=0.
\end{equation}
For the term $T_{2,1}$, we have
\begin{align*}
T_{2,1} &= \sum_{\sigma\in\E_{int}} \m(\sigma)\vv_{K,\sigma} \um_{\sigma,+}S_n(\um_{\sigma,+})(\varphi(x_K)-\varphi(x_L)) \\
&= \sum_{\sigma\in\E_{int}} \frac{\m(\sigma)d_{\sigma}}{d\,\m(D_{\sigma})}\,  \um_{\sigma,+}\,S_n(\um_{\sigma,+})\, d\frac{\varphi(x_K)-\varphi(x_L)}{d_{\sigma}}\int_{D_{\sigma}} \vv\cdot\eta_{K,\sigma}\,\mathrm{d}x \\
&=-\sum_{\sigma\in\E_{int}} \int_{D_{\sigma}} T_{2n}(\um_{\sigma,+})S_n(\um_{\sigma,+})\vv\cdot\nabla_{\M}\varphi_{\M}\,\mathrm{d}x.
\end{align*}
We define the function $\overline{G}_n^m$ defined over the diamonds by
\begin{equation*}
    \forall\sigma=K|L\in\E_{int},\forall x\in D_{\sigma}, \quad \overline{G}_n^m(x)=
    T_{2n}(\um_{\sigma,+})S_n(\um_{\sigma,+})
\end{equation*}
Then $T_{2,1}$ reads as
\begin{equation*}
    T_{2,1}= - \int_\Omega \overline{G}_n^m\vv\cdot\nabla_{{\M}_m}\varphi_{m}\,\mathrm{d}x.
\end{equation*}
Since the function $r\mapsto T_{2n}(r)S_n(r)$ is Lipschitz continuous and
bounded, with the help of arguments already used in the proof of Lemma
\ref{lemme6.1}, we can show that $\overline{G}_n^m$ converges to
$T_{2n}(u)S_n(u)$ in $L^\infty$ weak $\star$, as $h_{{\M}_{m}} \rightarrow 0$. Recalling
that $\nabla_{{\M}_{m}}\varphi_{m}$ converges weakly in $(L^2(\Omega))^d$ we obtain
\begin{equation}\label{eqt6.12}
    \lim_{h_{{\M}_{m}} \rightarrow 0} T_{2,1}= - \int_\Omega T_{2n}(u)S_n(u) \vv \cdot \nabla \varphi \, \mathrm{d}x.
    \end{equation}
We now  pass to the limit in the scheme first as $h_{{\M}_{m}}$ goes to zero and then as $n$ goes to infinity.
Gathering equations \eqref{cvT3} to \eqref{eqt6.12}, allows one to conclude that
\begin{multline}\label{limh_T}
\int_{\Omega}\lambda(u)S_n(u)\,\nabla u\cdot\nabla\varphi\,\mathrm{d}x -\int_{\Omega}u\,S_n(u)\,\vv\cdot\nabla\varphi\,\mathrm{d}x
- \int_{\Omega} f\,\varphi\,S_n(u)\,\mathrm{d}x = \lim_{h_{{\M}_{m}}\rightarrow 0} T(n,\varphi)
\end{multline}
where $\displaystyle\lim_{h_{{\M}_{m}}\rightarrow 0} |T(n,\varphi)| \le \|\varphi\|_{L^{\infty}(\Omega)} \omega(n)$ with $\omega(n)\rightarrow0$ as $n\rightarrow+\infty$. Since $S_n(u)\lambda(u)\nabla u$, $u S_n(u)\vv$ and $f S_n(u)$
belongs respectively to $(L^2(\Omega))^d$, $L^2(\Omega)$ and $L^1(\Omega)$ a density argument gives that \eqref{limh_T} holds true for any $\varphi$ lying in $H^1(\Omega)\cap L^\infty(\Omega)$.

Let $S$ be a function in $W^{1,\infty}(\mathbb{R})$ with compact support, contained in the interval $[-k,k]$, $k>0$ and let $\psi\in H^1(\Omega)\cap L^\infty(\Omega)$. Using the function $S(u)\psi$ in \eqref{limh_T}, we deduce that
\begin{multline*}
\bigg| \int_{\Omega}\lambda(u)\nabla u\, S_n(u)S(u)\nabla\psi\,\mathrm{d}x + \int_{\Omega}\lambda(u)\nabla u\, S_n(u)\psi\nabla u\, S'(u)\,\mathrm{d}x  \\
 -\int_{\Omega}u\,S_n(u)\,S(u)\,\vv\cdot\nabla\psi\,\mathrm{d}x - \int_{\Omega}u\,S_n(u)\,S'(u)\,\psi\,\vv\cdot\nabla u\,\mathrm{d}x \\
  - \int_{\Omega}\psi S(u) S_n(u) f\, \mathrm{d}x \bigg|
 \le \|\varphi\|_{L^{\infty}(\Omega)} \omega(n).
\end{multline*}
By observing that $S_n(u) S(u)=S(u)$ and $S_n(u)S'(u)=S'(u)$ a.e. in $\Omega$ for $n$ sufficiently large,  by passing to the limit as $n$ goes to infinity we obtain the condition \eqref{def3} of Definition \ref{def}, that is
\begin{multline*}
\int_{\Omega}\lambda(u)\nabla u\,S(u)\nabla\psi\,\mathrm{d}x + \int_{\Omega}\lambda(u)\nabla u\, \psi\nabla u\, S'(u)\,\mathrm{d}x  \\
-\int_{\Omega}u\,S(u)\,\vv\cdot\nabla\psi\,\mathrm{d}x - \int_{\Omega}u\,S'(u)\,\psi\,\vv\cdot\nabla u\,\mathrm{d}x  = \int_{\Omega}\psi S(u)\, f\, \mathrm{d}x.
\end{multline*}

We now turn to the decay of the energy. As a consequence of \eqref{energiediscrete} we get
\begin{equation*}
\lim_{n\rightarrow\infty}\varlimsup_{h_{{\M}_{m}}\rightarrow 0}\frac{1}{n}\sum_{\sigma\in\E_{int}}\lambda(u)\frac{\m(\sigma)}{d_{\sigma}}(T_{2n}(\um_K)-T_{2n}(\um_L))^2 = 0,
\end{equation*}
and
\begin{align*}
\sum_{\sigma\in\E_{int}}\frac{\m(\sigma)}{d_{\sigma}}(T_{2n}(\um_K)-T_{2n}(\um_L))^2 &= \sum_{\sigma\in\E_{int}}\m(\sigma)d_{\sigma}\left(\frac{T_{2n}(\um_K)-T_{2n}(\um_L)}{d_{\sigma}}\right)^2 \\
&= \sum_{\sigma\in\E_{int}}d \m(D_{\sigma})\left(\frac{T_{2n}(\um_K)-T_{2n}(\um_L)}{d_{\sigma}}\right)^2 \\
%&= \frac{1}{d}\sum_{\sigma\in\E}\m(D_{\sigma})\left(d\frac{T_{2n}(u_K)-T_{2n}(u_L)}{d_{\sigma}}\right)^2 \\
&= \frac{1}{d}\int_{\Omega}|\nabla_{{\M}_{m}} T_{2n}(u_{m})|^2\,\mathrm{d}x.
\end{align*}
so that
$\displaystyle\lim_{n\rightarrow\infty}\varlimsup_{h_{{\M}_{m}}\rightarrow
  0}\frac{1}{n}\int_{\Omega}|\nabla_{{\M}_{m}} T_{2n}(u_{m})|^2\,\mathrm{d}x =
0$. Since $\nabla_{{\M}_{m}} T_{2n}(u_{m})$ converges weakly in $(L^2(\Omega))^d$, we have
also
\begin{equation*}
\frac{1}{n}\int_{\Omega}|\nabla T_{2n}(u)|^2\,\mathrm{d}x \le \liminf_{h_{\M}\rightarrow 0}\frac{1}{n}\int_{\Omega}|\nabla_{\M} T_{2n}(u_{\M})|^2\,\mathrm{d}x,
\end{equation*}
which leads to
\begin{equation*}
\lim_{n\rightarrow\infty}\frac{1}{n}\int_{\Omega}\lambda(u)|\nabla T_{2n}(u)|^2\,\mathrm{d}x=0.
\end{equation*}
Since $u$ is finite almost everywhere in $\Omega$ and since $T_n(u)\in H^{1}(\Omega)$ for any $n>0$ we can conclude $u_{\M_m}$ converges to $u$ which is is the unique renormalized solution with null median.
\end{proof2}

\appendix
\section{Appendix}
Discrete functional inequalities are useful for the study of finite volume schemes. Discrete Sobolev inequalities are proved in \cite{refId0} for Dirichlet boundary conditions and in \cite{CHD11} for non Dirichlet boundary conditions. In \cite{zbMATH06476912} discrete Gagliardo-Nirenberg-Sobolev and Poincaré-Sobolev inequalities for some finite volume schemes are proved. There does not seem to be any proof of Discrete Poincaré-Wirtinger median inequality. The authors in \cite{zbMATH06476912} use the continuous embedding of the space $BV(\Omega)$ into $L^{\frac{d}{d-1}}(\Omega)$ for a Lipschitz domain $\Omega \subset \mathbb{R}^d$, with $d\geq 2$ to establish discrete inequalities. We will use this method to prove the Discrete Poincaré-Wirtinger median inequality \eqref{PW}.

Let us first recall some results concerning functions of bounded variation (more details about these functions can be found in \cite{Z}). Let $\Omega$ be an open set of $\R^d$ and $u\in L^1(\Omega)$. The total variation of $u$ in $\Omega$, denoted by $TV_\Omega(u)$, is defined by
$$ TV_\Omega(u)=\sup \left \{\int_\Omega u(x)\mathrm{div}(\phi(x))\,\mathrm{d}x,\; \phi \in C^1_c(\Omega),\;|\phi(x)|\leq 1, \; \forall x\in \Omega \right \}.$$
The function $u\in L^1(\Omega)$ belongs to $BV(\Omega)$ if and only if $TV_\Omega(u)<+\infty$. The space $BV(\Omega)$ is endowed with the norm
\[
  \|u\|_{BV(\Omega)}:=\|u\|_{L^1(\Omega)}+TV_{\Omega}(u).
\]
The space $BV(\Omega)$ is a natural space to study finite volume
approximations. Indeed, for $u=(u_{K})_{K\in\T}\in X(\mathcal{T})$, we
have
\[
  TV_\Omega(u)=\sum_{\underset{\sigma=K|L}{\sigma\in
      \E_{int}}}\m(\sigma)|u_L-u_K|=|u|_{1,1,\M}<+\infty.
\]
The discrete space
$X(\T)$ is included in $L^1(\Omega)\cap BV(\Omega)$ and we have
$$\|u\|_{BV(\Omega)}=\|u\|_{1,1,\M}.$$
%Let us now state the continuous embedding of $BV(\Omega)$ into $L^{\frac{d}{d-1}}(\Omega)$ which we will use in the proof later.
Our starting point for the discrete Poincaré-Wirtinger median inequality is the continuous embedding of $BV(\Omega)$ into $L^{\frac{d}{d-1}}(\Omega)$ for Lipschitz bounded connected domain $\Omega$ of $\R^{d}$, $d\ge2$, written in the following theorem (see \cite{Z} for more details).
\begin{thm}[\cite{Z}]\label{TBV}
There exists a constant $C(\Omega)>0$ only depending on $\Omega$ such that, for all $u\in BV(\Omega)$,
\begin{equation}\label{BV}
    \left(\int_\Omega|u-m|^{\frac{d}{d-1}}\,\mathrm{d}x \right)^{\frac{d-1}{d}}\leq C(\Omega) \,TV_\Omega(u),
\end{equation} where $m\in \med(u)$.
\end{thm}{}

In the spirit of \cite{zbMATH06476912} (which studied Discrete Poincaré-Wirtinger mean inequality), let us prove now the following proposition
 \begin{prop}[Discrete Poincaré-Wirtinger median inequality]\label{PWM} Let $\Omega$ be an open bounded connected polyhedral domain of $\R^d$ and let $\mathcal{M}$ be an admissible mesh satisfying \eqref{cmesh}. Then for $1\leq p <+\infty
  $ there exists a constant $C>0$ only depending on $\Omega$, $d$ and $p$ such that
  \begin{equation}
      \|u-c\|_{0,p}\leq \dfrac{C}{\xi^{(p-1)/p}}|u|_{1,p,\mathcal{M}}, \hspace*{1cm}\forall u\in X(\mathcal{T})
  \end{equation}{}
  where $c$ is in $\med(u)$.
  \end{prop}

  \begin{proof}[Proof of Proposition \ref{PWM}] \label{proofpc}
Let $u=(u_{K})_{K\in\T}$ be a function of $X(\T)$ and let $m$ be an element of
$\med(u)$. We define $v \in X(\T)$ by $v_K=(u_k-m)|u_k-m|^{p-1}$ for all $K\in
\T$. Since $m \in \med (u)$, we have $0 \in \med(v)$, using inequality
\eqref{BV}, we obtain
\begin{equation}
\|v\|_{0,\frac{d}{d-1}} \leq C(\Omega) |v|_{1,1,\M},
\end{equation}
and using the inclusion of $L^{\frac{d}{d-1}}(\Omega)$ into $L^1(\Omega)$, we get
\begin{equation}\label{A}
    \|v\|_{0,1} \leq C(\Omega,d)|v|_{1,1,\M},
\end{equation}
where the constant $C$ depends on $\Omega$ and $d$.\\
Moreover, for all $K, L \in \T$, we have
\begin{equation}
\begin{aligned}\label{B}
     |v_K-v_L|&=|u_K-u_L||v'(w_{LK})|, \hspace*{1cm} \forall w_{LK}\in[u_k,u_L],\\
     &\leq p|u_K-u_L||w_{LK}-m|^{p-1}\\
     &\leq p|u_K-u_L|\left (|u_K-m|^{p-1} + |u_L-m|^{p-1} \right).
\end{aligned}{}
\end{equation}{}
Therefore, gathering \eqref{A} and \eqref{B}, we obtain
\begin{equation}
\begin{aligned}{}
    \|v\|_{0,1}=\||u-m|^p\|_{0,1}&=\|u-m\|_{0,p}^p\\
    &\leq C \sum_{\underset{\sigma=K|L}{\sigma\in\mathcal{E}_{int}}}\m(\sigma)p|u_k-u_L|\left( |u_K-m|^{p-1} + |u_L-m|^{p-1}\right).
\end{aligned}
\end{equation}{}
Using Hölder's inequality we get,
\begin{equation}
\begin{aligned}{}
    \|u-m\|_{0,p}^p &\leq p C \left( \sum_{\underset{\sigma=K|L}{\sigma\in\mathcal{E}_{int}}}\frac{\m(\sigma)}{d_\sigma^{p-1}} |u_K-u_L|^p \right)^{\frac{1}{p}}\\ & \times \left(  \sum_{\underset{\sigma=K|L}{\sigma\in\mathcal{E}_{int}}} \m(\sigma)\left(d_\sigma^{\frac{p-1}{p}}\right)^{\frac{p}{p-1}}\left( |u_K-m|^{p-1} + |u_L-m|^{p-1}\right)^{\frac{p}{p-1}}\right)^{\frac{p-1}{p}}\\
    &\leq p C |u-m|_{1,p,\M} \left(\sum_{\underset{\sigma=K|L}{\sigma\in\mathcal{E}_{int}}} \m(\sigma)d_\sigma\dfrac{p}{p-1}\left(|u_K-m|^p+ |u_L-m|^p\right) \right)^{\frac{p-1}{p}}.
    \end{aligned}
\end{equation}{}
The regularity constraint \eqref{cmesh} on the mesh ensures that
\begin{equation}\label{C}
    \sum_{\sigma\in \E_{int}}\m(\sigma)d_\sigma\leq \frac{1}{\xi} \sum_{K\in \T}
    \sum_{\sigma\in \E_K}\m(\sigma)d(x_K,\sigma)=\frac{N}{\xi}\sum_{K\in
      \T}\m(K).
\end{equation}{}
Then applying  the previous inequality \eqref{C} and a discrete integration by parts, we get
\begin{equation}
\begin{aligned}{}
    \|u-m\|_{0,p}^p & \leq p C(\Omega)
    |u-m|_{1,p,\M}\left(\frac{p}{p-1}\right)^{\frac{p-1}{p}}\left(\sum_{K\in\T}\sum_{\sigma\in\E_K}\m(\sigma)d_\sigma|u_K-m|^p\right)^{\frac{p-1}{p}}\\
    & \leq \frac{ C(\Omega,p,d)}{\xi^{\frac{p-1}{p}}}|u-m|_{1,p,\M}\|u-m\|_{0,p}^{p-1}.
    \end{aligned}
\end{equation}{}
Then we obtain the general result
\begin{equation*}
    \|u-m\|_{0,p}\leq \dfrac{C}{\xi^{(p-1)/p}}|u|_{1,p,\mathcal{M}}, \hspace*{1cm}\forall u\in X(\mathcal{T}).
\end{equation*}
\end{proof}

\bibliography{biblio.bib}

\bibliographystyle{acm}

\end{document}